\documentclass[12pt]{article}
\usepackage{amsmath,amsthm,amssymb,euscript,verbatim}
\newtheorem{theorem}{Theorem}[section]

\newtheorem{corollary}{Corollary}[section]
\newtheorem{remark}{Remark}[section]

\usepackage{makeidx}         
\usepackage{graphicx}        
\usepackage{multicol}        
\usepackage[bottom]{footmisc}

\makeindex             
\usepackage{latexsym}
\usepackage{epic,eepic,pstricks}

\begin{document}

\title
{\bf  On the Work of Cartan and M\"{u}nzner on Isoparametric Hypersurfaces}
\author
{Thomas E. Cecil and Patrick J. Ryan}
\maketitle

\begin{abstract}
A hypersurface $M^n$ in a real space form ${\bf R}^{n+1}$, $S^{n+1}$, or $H^{n+1}$ is isoparametric if it has
constant principal curvatures.
This paper is a survey of the fundamental work of Cartan
and M\"{u}nzner  on the theory of isoparametric hypersurfaces in real space forms, in particular, spheres.
This work is contained in four papers of Cartan \cite{Car2}--\cite{Car5} published during the period
1938--1940, and two papers of M\"{u}nzner \cite{Mu}--\cite{Mu2} that were published in preprint form in the early 1970's,
and as journal articles in 1980--1981.  These papers of Cartan and M\"{u}nzner have been the foundation of the extensive field of isoparametric hypersurfaces, and they
have all been recently translated into English by T. Cecil.
The paper concludes with a brief survey of the recently completed classification of isoparametric hypersurfaces in spheres
in Section \ref{chap:2}.
\end{abstract}

\section{Introduction}
\label{sec-introduction} 
A hypersurface $M^n$ in a real space form ${\bf R}^{n+1}$, $S^{n+1}$, or $H^{n+1}$ is isoparametric if it has
constant principal curvatures.
This paper is a survey of the fundamental work of Cartan
and M\"{u}nzner  on the theory of isoparametric hypersurfaces in real space forms.
During the period 1938--1940, Cartan \cite{Car2}--\cite{Car5} published four papers that classified isoparametric hypersurfaces in  ${\bf R}^{n+1}$ and $H^{n+1}$, and laid the groundwork for the theory of isoparametric
hypersurfaces in the sphere $S^{n+1}$,  especially in the case where the principal curvatures of the hypersurface all have the same multiplicity.

Approximately thirty years later in the early 1970's, M\"{u}nzner \cite{Mu}--\cite{Mu2} published two preprints
which extended some of Cartan's main results to the arbitrary case where there are no assumptions on the multiplicities
of the principal curvatures.
M\"{u}nzner then proved his major result that the number 
$g$ of distinct principal curvatures of an isoparametric hypersurface
in a sphere satisfies the restriction $g = 1, 2, 3, 4$ or 6.  M\"{u}nzner's articles were later published as journal articles in 1980--1981.

These papers of Cartan and M\"{u}nzner have proven to be the foundation for the
extensive theory of isoparametric hypersurfaces, which has been developed by many researchers over the years.  
The classification of isoparametric hypersurfaces in spheres has been recently completed
(see Chi \cite{Chi-survey}), 
and this paper concludes with a brief survey of the classification results in Section \ref{chap:2}.

T. Cecil has recently published English translations of all six of these papers of Cartan
and M\"{u}nzner at CrossWorks, the publishing site for the College of the Holy Cross. 
References for these translations are listed 
together with the references for 
the papers of Cartan and  M\"{u}nzner in the bibliography.  

In general, 
this paper follows the treatment of the theory of isoparametric hypersurfaces in the book by the authors
\cite[pp. 85--137]{CR8}, and some passages are taken directly from that book.  Several proofs of theorems
are omitted in this paper, but they are included in the book \cite{CR8}, and detailed references are given
for the proofs.

The authors wish to acknowledge valuable collaborations and conversations on the theory of isoparametric hypersurfaces with Quo-Shin Chi, Gary Jensen, and Miguel Dominguez-Vazquez.

\section{Cartan's Work}
\label{Cartan-work}

\subsection{Families of Isoparametric Hypersurfaces}
\label{sec-families-isop} 

We begin with some preliminary notation and definitions, which generally follow those in the book
\cite{CR8}.
By a real {\em space form}
of dimension $n$, we mean a complete, connected, simply connected manifold $\widetilde{M}^n (c)$ with constant sectional curvature $c$.  If $c=0$, then $\widetilde{M}^n (c)$ is the $n$-dimensional Euclidean space ${\bf R}^n$; if
$c=1$, then $\widetilde{M}^n (c)$ is the unit sphere $S^n \subset {\bf R}^{n+1}$; and if $c= -1$, then $\widetilde{M}^n (c)$ is the $n$-dimensional real hyperbolic space $H^n$ (see, for example, \cite[Vol. I, pp. 204--209]{KN}).  For any value of $c>0$, this subject is basically the same as for the case $c=1$, and for any value of $c<0$, it is very similar to the case $c=-1$.  So we will restrict our attention to the cases $c = 0, 1, -1$.

The original definition of a family of isoparametric hypersurfaces in a real space
form $\widetilde{M}^{n+1}$ was formulated in terms of the level sets of an isoparametric function, as we now describe. 
Let $F:\widetilde{M}^{n+1} \rightarrow {\bf R}$ be a nonconstant smooth function. The classical Beltrami
differential parameters of $F$ are defined by
\begin{equation}
\label{eq:1-beltrami}
\Delta_1 F = |{\rm grad}\  F|^2, \quad \Delta_2 F = \Delta F\ ({\rm Laplacian}\ {\rm of}\  F),
\end{equation}
where grad $F$ denotes the gradient vector field of $F$.

The function $F$ is said to be {\em isoparametric} if there exist smooth functions $\phi_1$ and $\phi_2$ from ${\bf R}$ to ${\bf R}$ such that
\begin{equation}
\label{eq:1-beltrami-isoparametric}
\Delta_1 F = \phi_1 (F), \quad \Delta_2 F = \phi_2 (F).
\end{equation}
That is, both of the Beltrami differential parameters are constant on each level set of $F$. This is the origin of the term {\em isoparametric}.
The collection of level sets of an isoparametric function is called an {\em isoparametric family} of sets in
$\widetilde{M}^{n+1}$.

An isoparametric family in ${\bf R}^{n+1}$ consists of either
parallel planes, concentric spheres, or coaxial spherical cylinders,
and their focal sets.
This was first shown for $n=2$ by Somigliana \cite{Som} 
(see also B. Segre \cite{Seg1} and Levi-Civita  \cite{Lev}), and for arbitrary $n$ by B. Segre \cite{Seg}.

In the late 1930's, shortly after the publication of the papers of 
Levi-Civita and Segre,  Cartan \cite{Car2}--\cite{Car5} began a study of isoparametric families in arbitrary
real space forms. 
Cartan began his first paper \cite{Car2} with an extract of a letter that he had recently written to B. Segre.  Cartan states that he can extend Segre's results for isoparametric families in ${\bf R}^{n+1}$ to isoparametric families in hyperbolic space $H^{n+1}$, and that he has made progress in the study of isoparametric families in the sphere $S^{n+1}$.
Cartan notes that he employs a different technique than the one that Segre used.  Cartan uses the method of moving frames and the theory of parallel hypersurfaces, which we will briefly review here.

\subsection{Parallel Hypersurfaces}
\label{sec-isop-parallel-hypersurfaces}

Let $f:M^n \rightarrow \widetilde{M}^{n+1} (c)$ be an oriented hypersurface with field of unit normals $\xi$.  
For $x \in M^n$, let $T_x M^n$ denote
the tangent space to $M^n$ at $x$.  For any vector $X$ in the tangent space $T_x M^n$, we have the fundamental
equation
\begin{equation}
\label{eq:fundamental-equation}
\widetilde{\nabla}_{f_*(X)} \xi = -f_* (A_\xi X),
\end{equation}
where $\widetilde{\nabla}$ is the Levi-Civita connection in $\widetilde{M}$, $f_*$ is the differential of $f$, and
$A_\xi$ is the {\em shape operator} determined by the normal vector field $\xi$.

The shape operator $A_\xi$ is a
symmetric tensor of type $(1,1)$ on $M^n$.  An eigenvalue $\lambda$ of $A_\xi$ is called a {\em principal curvature}
of $M^n$, and a corresponding eigenvector is called a {\em principal vector}.  We often write $A$ instead of 
$A_\xi$, in the case where the choice of field of unit normals $\xi$ has been specified.

The {\em parallel hypersurface} to $f(M^n)$ at signed distance $t \in {\bf R}$ is the map $f_t: M^n \rightarrow \widetilde{M}^{n+1} (c)$ such that for each $x \in M^n$, the point $f_t (x)$ is obtained by traveling a signed distance $t$ along the geodesic in 
$\widetilde{M}^{n+1} (c)$ with initial point $f(x)$ and initial tangent vector $\xi(x)$.  For 
$\widetilde{M}^{n+1} (c) = {\bf R}^{n+1}$, the formula for $f_t$ is
\begin{equation}
\label{eq:parallel-map}
f_t (x) = f(x) + t \ \xi (x),
\end{equation}
and for $\widetilde{M}^{n+1} (c) = S^{n+1}$, the formula for $f_t$ is
\begin{equation}
\label{eq:parallel-map-sphere}
f_t (x) = \cos t \ f(x) + \sin t \ \xi (x).
\end{equation}
There is a similar formula in hyperbolic space $H^{n+1}$ (see, for example, \cite{Cec1}).

Locally, for sufficiently small values of $t$, the map $f_t$ is also an immersed hypersurface.  However,
the map $f_t$ may develop singularities at focal points of the original hypersurface $f(M^n)$.  Specifically, a point 
$p = f_t (x)$ is called a {\em focal point of $(M^n,x)$ of multiplicity $m > 0$} if the differential $(f_t)_*$ has nullity $m$ at $x$.

In the cases $\widetilde{M}^{n+1} (c) = {\bf R}^{n+1}$, respectively $S^{n+1}$,
the point $p = f_t (x)$ is a focal point of $(M^n,x)$ of 
multiplicity $m > 0$ if and only if $1/t$, respectively $\cot t$, is a principal curvature of $M^n$ of multiplicity $m$ at $x$.
Thus, in the case $\widetilde{M}^{n+1} (c) = S^{n+1}$, each principal curvature $\lambda$ gives rise to two antipodal focal 
points along the normal geodesic, since $\lambda = \cot t = \cot (t + \pi)$ (see, for example, \cite[p. 11]{CR8}).
There is a similar formula for focal points in the ambient space $H^{n+1}$ (see, for example, \cite{Cec1} or \cite[p. 11]{CR8}).

Now let $f:M^n \rightarrow \widetilde{M}^{n+1} (c)$ be an oriented hypersurface  with constant principal curvatures.  One can show from the formulas for the principal curvatures of a parallel hypersurface 
that if $f$ has constant principal curvatures, then each immersed hypersurface $f_t$ also has constant principal curvatures (see, for example, \cite[pp. 17--18]{CR8} or Subsection \ref{sec:1.2} of this paper).

Since the principal curvatures are constant on $M^n$, the focal points along the normal geodesic to $f(M^n)$ at $f(x)$ occur for the same values of $t$, independent of the choice of the point 
$x \in M^n$.  
In the case of ${\bf R}^{n+1}$, for example, this means that if $\mu$ is a nonzero constant principal curvature of multiplicity $m>0$ of $M^n$, then the map $f_{1/\mu}$ has constant rank $n-m$ on $M^n$, and the set $f_{1/\mu} (M^n)$ is an 
$(n-m)$-dimensional  submanifold of ${\bf R}^{n+1}$, called a {\em focal
submanifold} of $f(M^n)$.  Similar results hold for the other ambient space forms
(see, for example, \cite[pp. 18--32]{CR8}).

\subsection{Isoparametric and Parallel Hypersurfaces}
\label{sec-equivalence-isop-parallel-hypersurfaces}

In Section 2 of his first paper,  Cartan \cite{Car2} showed that an isoparametric family, defined as the collection of level sets of an isoparametric function $F$, is equivalent to a family of parallel hypersurfaces, each of which has constant principal curvatures, together with their focal submanifolds.  

We now discuss the main steps in this argument.
We will omit most of the proofs, and we refer the reader to  \cite[pp. 86--90]{CR8} for detailed proofs of these steps,  which are contained in Theorems \ref{thm:1-mean-curv-level-set} -- \ref{cor:1-constant-prin-curv-parallel-hyp} below.

Let $F:\widetilde{M}^{n+1} \rightarrow {\bf R}$ be a nonconstant smooth function defined on a real space form
$\widetilde{M}^{n+1}$.  Suppose that ${\rm grad}\  F$ does not vanish on the level set $M = F^{-1} (0)$.  Then 
$M$ is a smooth hypersurface in $\widetilde{M}^{n+1}$, and the shape operator $A$
of $M$ satisfies the equation
\begin{equation}
\label{eq:1-shape-oper-hessian}
\langle AX, Y \rangle = - \frac{H_F (X,Y)}{|{\rm grad}\  F|},
\end{equation}
where $X$ and $Y$ are tangent vectors to $M$, and $H_F$ is the Hessian of the function $F$.

Then one can show by a calculation, that the mean curvature $h = ({\rm trace}\ A)/n$ of the level  hypersurface $M$ is given by the following theorem.

\begin{theorem}
\label{thm:1-mean-curv-level-set} 
The mean curvature $h$ of the level hypersurface $M$ is given by
\begin{equation}
\label{eq:1-mean-curv-level-set}
h = \frac{1}{n \rho^2} ( \langle {\rm grad}\  F,\ {\rm grad}\  \rho \rangle - \rho \Delta F),
\end{equation}
where $\rho =  |{\rm grad}\  F|$.
\end{theorem}

The following theorem  of Cartan is important in showing that the level sets of an isoparametric function have constant 
principal curvatures.

\begin{theorem}
\label{thm:1-level-set-constant-mean-curv} 
If $F:\widetilde{M}^{n+1} \rightarrow {\bf R}$ is an isoparametric function on a real space form, then each level
hypersurface of $F$ has constant mean curvature.
\end{theorem}

The vector field $\xi = {\rm grad}\  F/ |{\rm grad}\  F|$ is defined on the open subset of $\widetilde{M}^{n+1}$ on which
${\rm grad}\  F$ is nonzero.  A direct calculation then yields the following theorem that the integral curves of $\xi$ are geodesics in $\widetilde{M}^{n+1}$.

\begin{theorem}
\label{thm:1-integral-curves-xi} 
Let $F:\widetilde{M}^{n+1} \rightarrow {\bf R}$ be a function for which $|{\rm grad}\  F|$ is a function of $F$.  Then
on the subset of $\widetilde{M}^{n+1}$ where $|{\rm grad}\  F|$ is nonzero,
the integral curves of the vector field $\xi = {\rm grad}\  F/ |{\rm grad}\  F|$ are geodesics in $\widetilde{M}^{n+1}$.
\end{theorem}

Theorem \ref{thm:1-integral-curves-xi} shows that a family of level hypersurfaces of an isoparametric function
is a family of parallel hypersurfaces in $\widetilde{M}^{n+1}$, modulo reparametrization to take into account the
possibility that $|{\rm grad}\  F|$ is not identically equal to one.  The next step is to show that each of these level
hypersurfaces has constant principal curvatures.  This follows from Theorem \ref{thm:1-level-set-constant-mean-curv}
and the next theorem of Cartan.

\begin{theorem}
\label{thm:1-constant-mean-curv-parallel-hyp} 
Let $f_t:M \rightarrow \widetilde{M}^{n+1}$, $-\varepsilon < t < \varepsilon$, be a family of parallel hypersurfaces
in a real space form.  Then $f_0 M$ has constant principal curvatures if and only if each $M_t = f_t M$ has constant
mean curvature.
\end{theorem}

As a consequence of Theorems \ref{thm:1-level-set-constant-mean-curv}--\ref{thm:1-constant-mean-curv-parallel-hyp},
we obtain the following theorem of Cartan.

\begin{theorem}
\label{cor:1-constant-prin-curv-parallel-hyp}
If $F:\widetilde{M}^{n+1} \rightarrow {\bf R}$ is an isoparametric function on a real space form, then each level
hypersurface of $F$ has constant principal curvatures. 
\end{theorem}

Conversely, let $f_t:M \rightarrow \widetilde{M}^{n+1}$, $-\varepsilon < t < \varepsilon$, 
be a family of parallel hypersurfaces such that $f_0$
has constant principal curvatures.  We see that this is an isoparametric family of hypersurfaces as follows.

By the well known formulas for the principal curvatures of a parallel hypersurface in a real space form (see, for example, \cite[pp. 17--18]{CR8} and Subsection \ref{sec:1.2} of this paper), each $f_t M$ has constant principal curvatures,
and thus each $f_t M$ has constant mean curvature.  Then the function $F$ defined by
$F(x) = t$, if $x \in f_t M$, is a smooth function defined on an open subset of $\widetilde{M}^{n+1}$ with the property that
${\rm grad}\  F = \xi$ is a unit length vector field such that $\widetilde{\nabla}_{\xi} \xi = 0$. 

Furthermore,
since the function $\rho = |{\rm grad}\  F| = 1$ is constant, we see from Theorem \ref{thm:1-mean-curv-level-set} 
that the constancy of the mean
curvature $h$ on each level hypersurface $f_t M$ implies that the Laplacian of $F$ is also constant
on each level set $f_t M$, and therefore $F$ is an isoparametric function, as defined by  equation 
\eqref{eq:1-beltrami-isoparametric}. Thus, this family of parallel hypersurfaces is a family of isoparametric hypersurfaces in the original sense, that is, it is the family of level sets of an isoparametric function.

In summary, the analytic definition of an isoparametric family of hypersurfaces in terms of level sets of an isoparametric function on a real space form $\widetilde{M}^{n+1}$ 
is equivalent to the geometric definition of an isoparametric family as a
family of parallel hypersurfaces to a hypersurface 
with constant principal curvatures.

Therefore, in the case where the ambient space is a real space form $\widetilde{M}^{n+1}$,
we will say that  a connected hypersurface $M^n$ immersed in $\widetilde{M}^{n+1}$ is an {\em isoparametric hypersurface} if it has constant principal curvatures.

\begin{remark}
\label{rem-isop-complex-space-forms}
{\rm The definition of an isoparametric hypersurface in the paragraph above is not the appropriate formulation if 
$\widetilde{M}$  is only assumed to be a Riemannian manifold.  Rather the appropriate generalization is the following: a hypersurface $M$ in a Riemannian manifold $\widetilde{M}$ is {\em isoparametric} if and only if all nearby parallel hypersurfaces $M_t$ have constant mean curvature
(see Theorem \ref{thm:1-constant-mean-curv-parallel-hyp}). 

For example, when $\widetilde{M}$ is complex $n$-dimensional projective space ${\bf CP}^n$ or complex
$n$-dimensional hyperbolic space ${\bf CH}^n$, a
(real) isoparametric hypersurface may have nonconstant principal curvatures.  Nevertheless, making use of the Hopf map $\pi: S^{2n+1} \rightarrow {\bf CP}^n$ (see \cite[p. 346]{CR8}), one can show that an isoparametric hypersurface in $ {\bf CP}^n$ lifts to an isoparametric hypersurface in $S^{2n+1}$.  This facilitates the classification of isoparametric hypersurfaces in  ${\bf CP}^n$.

In fact, similar methods can be used to classify isoparametric hypersurfaces in ${\bf CH}^n$, although in that case one
must deal with Lorentzian isoparametric hypersufaces in the anti-de Sitter space.
For further details about isoparametric hypersurfaces in complex space forms, see Q.-M. Wang \cite{Wang-1a}, G. Thorbergsson \cite{Th7}, L. Xiao \cite{Xiao} and M. Dominguez-Vazquez et al. \cite{DV2}--\cite{DV3}.}
\end{remark}

\subsection{Cartan's Formula}
\label{sec-cartan-formula} 

In Section 3 of his first paper on isoparametric hypersurfaces, Cartan \cite{Car2} derived what he called a
``fundamental formula,'' which is the basis for many of his results on the subject.

The formula involves  the distinct
principal curvatures $\lambda_1,\ldots,\lambda_g$, and their respective multiplicities $m_1,\ldots,m_g$, of an isoparametric
hypersurface $f:M^n \rightarrow \widetilde{M}^{n+1}(c)$ in a space form of constant sectional
curvature $c$.
If the number $g$ of distinct principal curvatures is greater than one, Cartan showed that for each $i$,
$1 \leq i \leq g$, the following equation holds,
\begin{equation}
\label{eq:1-car-id}
\sum_{j\neq i} m_j \frac{c + \lambda_i \lambda_j}{\lambda_i - \lambda_j} = 0.
\end{equation}
Cartan's original formulation of the formula was given in a more complicated form, but he gave the formulation in
equation \eqref{eq:1-car-id} in his third paper \cite[p. 1484]{Car4}.  Equation \eqref{eq:1-car-id} is now known as 
``Cartan's formula.''  Another proof of this formula is given in \cite[pp. 91--96]{CR8}.

\begin{remark}
\label{rem-cartan-formula-in-complex-space-forms}
{\rm For a generalization of Cartan's formula to the case where $M$ is a Hopf hypersurface in complex projective space
${\bf CP}^n$ or in complex hyperbolic space ${\bf CH}^n$, see \cite[pp. 422--427]{CR8}.}
\end{remark}

\subsection{Isoparametric Hypersurfaces in ${\bf R}^{n+1}$}
\label{sec-iso-hyp-euc-space}

Using his formula, Cartan was able to completely classify isoparametric hypersurfaces in Euclidean space
${\bf R}^{n+1}$ $(c=0)$ and in hyperbolic space $H^{n+1}$ $(c= -1)$, and to make significant progress on the theory of isoparametric hypersurfaces in the sphere $S^{n+1}$ $(c= 1)$, as we will discuss below.

We first consider an isoparametric hypersurface in Euclidean space.
The classification is a local theorem, so we consider the hypersurface to be embedded in ${\bf R}^{n+1}$.  Here
$S^k (r)$ denotes a $k$-dimensional sphere of radius $r$
in a totally geodesic ${\bf R}^{k+1} \subset {\bf R}^{n+1}$. 

\begin{theorem}
\label{thm:isop-hyp-R} 
Let $M^n \subset {\bf R}^{n+1}$ be a connected isoparametric hypersurface.  Then $M^n$ is an open subset of a flat
hyperplane, a metric hypersphere, or a spherical cylinder $S^k (r) \times {\mathbf R}^{n-k}$.
\end{theorem}

\begin{proof}  
If the number $g$ of distinct principal curvatures of $M^n$ is one, 
then $M^n$ is totally umbilic, and it is well known that $M^n$
is an open subset of a hyperplane or hypersphere in ${\bf R}^{n+1}$
(see, for example, \cite[Vol. 4, p. 110]{Sp}).

If $g \geq 2$, then by taking an appropriate choice of unit normal field $\xi$, one
can assume that at least one of the principal curvatures is positive.  If $\lambda_i$ is the smallest positive
principal curvature, then each term $\lambda_i \lambda_j/(\lambda_i - \lambda_j)$
in Cartan's formula (\ref{eq:1-car-id}) with $c=0$ is non-positive, and thus it equals zero.
Hence, there can be at most two distinct principal curvatures, and if there are two, then one of them equals zero.

If $g = 2$, then by an appropriate choice of the unit normal field $\xi$,
we can arrange that the nonzero principal curvature is positive.
Thus, suppose that the principal curvatures are $\lambda_1 >0$ with multiplicity $m_1 = k$, and
$\lambda_2 = 0$ with multiplicity $m_2 = n-k$.  Then for $t = 1/\lambda_1$, the focal submanifold $V = f_t (M^n)$ 
has dimension $n-k$, and it is totally geodesic in ${\bf R}^{n+1}$, since the formulas for the shape operator of a parallel map $f_t$ \cite[pp. 17-18]{CR8} show that 
for every unit normal $\eta$ to $V$ at every point $p$ of $V$,
the shape operator $A_\eta$ has one distinct principal curvature given by 
\begin{equation}
\label{eq:1-foc-sub-tot-geod}
\frac{\lambda_2}{(1 - t \lambda_2)} = 0.
\end{equation}

Thus, $V$ is contained in a totally geodesic submanifold ${\bf R}^{n-k} \subset {\bf R}^{n+1}$,
and $M^n$ is an open subset of a tube of radius $1/ \lambda_1$ over ${\bf R}^{n-k}$. Such a tube is a spherical cylinder 
$S^k (r) \times {\mathbf R}^{n-k}$, where $S^k (r)$ is a $k$-dimensional sphere of radius $r=1/\lambda_1$
in a totally geodesic ${\bf R}^{k+1} \subset {\bf R}^{n+1}$ orthogonal to ${\bf R}^{n-k}$.
\end{proof}

\subsection{Isoparametric Hypersurfaces in $H^{n+1}$}
\label{sec-iso-hyp-hyperbolic-space}

Next we consider an isoparametric hypersurface $M^n \subset H^{n+1}$ in hyperbolic space $(c = -1)$.
As we will see in the proof of Theorem \ref{thm:isop-hyp-h} below,
Cartan used his formula again to show that the number $g$ of distinct principal curvatures must be less than or equal to 2. The rest of the proof then follows in a way similar to the Euclidean case above.
Again this is a local result, so we consider the hypersurface to be embedded. 

\begin{theorem}
\label{thm:isop-hyp-h} 
Let $M^n \subset {\bf H}^{n+1}$ be a connected isoparametric hypersurface.  Then $M^n$ is 
an open subset of a totally geodesic hyperplane, an equidistant hypersurface, a horosphere, a 
metric hypersphere, or a tube over a totally geodesic submanifold of codimension greater than one in $H^{n+1}$.
\end{theorem}
 
\begin{proof}
Let $g$ be the number of distinct principal curvatures of $M^n$.
If $g = 1$, then $M^n$ is totally umbilic, and it
is an open subset of a totally geodesic hyperplane, an equidistant hypersurface,
a horosphere, or a metric hypersphere
in $H^{n+1}$ (see, for example, \cite[Vol. 4, p. 114]{Sp}). 

If $g \geq 2$, then by an appropriate choice of the unit normal field $\xi$,
we can arrange that at least one of the principal curvatures is positive.  Then 
there exists a positive principal curvature $\lambda_i$ such that no principal curvature
lies between $\lambda_i$ and $1/ \lambda_i$. (In fact, $\lambda_i$ is either the largest principal curvature 
between 0 and 1, or the smallest principal curvature greater than or equal to one.) 

For this $\lambda_i$,
each term $(-1 +\lambda_i \lambda_j)/(\lambda_i - \lambda_j)$
in Cartan's formula (\ref{eq:1-car-id}) with $c=-1$ is negative unless $\lambda_j = 1/\lambda_i$, in which case the term is zero. Therefore, there can be at most two distinct
principal curvatures, and if there are two, they are reciprocals of each other.  

In the case $g=2$, suppose that the two principal curvatures are $\lambda_1 = \coth \theta$ with multiplicity $m_1 = k$, and 
$\lambda_2 = 1/\lambda_1 = \tanh \theta$ with multiplicity $m_2 = n-k$.  If we take $t = \theta$, then $V = f_t (M^n)$
is a focal submanifold of dimension $n-k$, and it is totally geodesic in $H^{n+1}$ by the formulas for the shape operator of a parallel map $f_t$ \cite[pp. 17-18]{CR8}.

Thus, this focal submanifold $V$ is contained in a totally geodesic submanifold $H^{n-k} \subset H^{n+1}$,
and $f(M^n)$ is an open subset of a tube of radius $t = \theta$ over $H^{n-k}$. Such a tube is standard Riemannian
product $S^k (c_1) \times H^{n-k} (c_2)$ in hyperbolic space $H^{n+1}$, where $c_1 = 1/\sinh^2 \theta$
and $c_2 = -1/ \cosh^2 \theta$ are the constant sectional curvatures of the sphere $S^k (c_1)$ 
and the hyperbolic space $H^{n-k} (c_2)$, respectively (see Ryan \cite{Ryan3}).
\end{proof}

\subsection{Isoparametric Hypersurfaces in $S^{n+1}$}
\label{sec-iso-hyp-in-spheres}

In the sphere $S^{n+1}$, however, Cartan's formula does not lead to the conclusion that $g \leq 2$.  In fact, Cartan
produced examples with $g = 1, 2, 3$ or 4 distinct
principal curvatures.  Moreover, he classified isoparametric hypersurfaces $M^n \subset S^{n+1}$ with
$g \leq 3$ in \cite{Car3} (see Section \ref{cartan-case g=3}). 

In the case $g=1$, the hypersurface $M^n$ is totally umbilic, and it is well known that $M^n$
is an open subset of a great or small hypersphere in $S^{n+1}$ (see \cite[Vol. 4, p. 112]{Sp}).  

If $g=2$, then $M^n$ is an open subset of a standard product of two spheres,
\begin{equation}
\label{product}
S^p (r) \times S^q (s) \subset S^{n+1} (1) \subset {\bf R}^{p+1} \times {\bf R}^{q+1} = {\bf R}^{n+2}, \quad r^2 + s^2 = 1,
\end{equation}
where $n = p+q$, and $r>0$, $s>0$ are the radii of the respective spheres. The proof of this result is similar to the proofs of Theorems \ref{thm:isop-hyp-R} and \ref{thm:isop-hyp-h} above, i.e., one shows that a focal submanifold is totally geodesic, and $M^n$ is a tube of constant radius over it. This is true for both focal submanifolds in this case
(see, for example, \cite[pp. 110--111]{CR8}).

The case of $g=3$ distinct principal curvatures is much more difficult, and Cartan's thorough treatment of that case 
in \cite{Car3} is a highlight of his work on isoparametric hypersurfaces.

Specifically, Cartan \cite{Car3}
showed that in the case $g=3$, all the principal curvatures must have the same multiplicity $m$, which must be
one of the values $m=1,2,4$ or 8.  Furthermore, $M^n$ must be an open subset of a tube of
constant radius over a standard embedding of a projective
plane ${\bf FP}^2$ into $S^{3m+1}$, 
where ${\bf F}$ is the division algebra
${\bf R}$, ${\bf C}$, ${\bf H}$ (quaternions),
${\bf O}$ (Cayley numbers), for $m=1,2,4,8,$ respectively. 

Thus, up to congruence, there is only one such family of isoparametric hypersurfaces
for each value of $m$.  For each of these hypersurfaces,
the focal set of $M^n$ consists of two antipodal standard embeddings of ${\bf FP}^2$,
and $M^n$ is a tube of constant radius over each focal submanifold.
We will describe Cartan's work in the case $g=3$ in more detail in
Subsection \ref{cartan-case g=3} below.

\subsection{The case where all multiplicities are equal}
\label{sec:equal-multiplicities}

Cartan \cite[pp. 186--187]{Car2} gave a method for determining the values of the principal curvatures of 
an isoparametric hypersurface $M^n$, if one knows the value of $g$ and the multiplicities $m_1,\ldots,m_g$, as in equation \eqref{eq:1-prin-curv-sph}.
This is method is based on 
an analytic calculation involving Cartan's formula \eqref{eq:1-car-id}.
In the case where all of the principal curvatures have the same multiplicity $m = n/g$, Cartan's method 
\cite[pp. 34--35]{Car4} yields the formula for the principal curvatures given in Theorem 
\ref{thm:1-prin-curv-isop-hyp-cartan} below.

This formula implies that for any point $x \in M$, there
are $2g$ focal points of $(M,x)$ along the normal geodesic to $M$ through $x$, and they are evenly
distributed at intervals of length $\pi/g$.

\begin{theorem} [Cartan]
\label{thm:1-prin-curv-isop-hyp-cartan}  
Let $M \subset S^{n+1}$ be a connected isoparametric hypersurface with $g$ principal curvatures
$\lambda_i = \cot \theta_i$, $0 < \theta_1 < \cdots < \theta_g < \pi$, all having the same multiplicity $m = n/g$. Then
\begin{equation}
\label{eq:1-prin-curv-formula-1}
\theta_i = \theta_1 + (i-1) \frac{\pi}{g} , \quad 1 \leq i \leq g.
\end{equation} 
For any point $x \in M$, there
are $2g$ focal points of $(M,x)$ along the normal geodesic to $M$ through $x$, and they are evenly
distributed at intervals of length $\pi/g$.
\end{theorem}

\begin{remark}
\label{rem-munzner-thm}
{\rm M\"{u}nzner \cite{Mu}
later showed that formula \eqref{eq:1-prin-curv-formula-1} is valid for all isoparametric hypersurfaces in
$S^{n+1}$, even those where the multiplicities are not all the same. M\"{u}nzner's method is different than
the method of Cartan.
M\"{u}nzner's proof is based on the fact
that the set of focal points along a normal geodesic circle to $M \subset S^{n+1}$ is invariant under the 
dihedral group $D_g$
of order $2g$ that acts on the normal circle and is generated by reflections in the focal points.
M\"{u}nzner's proof also yields the important result that
the multiplicities satisfy $m_i = m_{i+2}$ (subscripts mod $g$). We will give a detailed
proof of M\"{u}nzner's  result in Theorem \ref{thm:1-prin-curv-isop-hyp}
in Subsection \ref{formula-p-c} below.}
\end{remark}

Using a lengthy calculation involving equation \eqref{eq:1-prin-curv-formula-1}, Cartan \cite[pp. 364-367]{Car3}
proved the important result
that any isoparametric family with $g$
distinct principal curvatures all having the same multiplicity is algebraic in sense of Theorem \ref{thm:1-Cartan-1} 
below.

Recall that a
function $F:{\bf R}^{n+2} \rightarrow {\bf R}$ is {\em homogeneous of degree} 
$g$ if $F(tx) = t^g F(x)$, for all
$t \in {\bf R}$ and $x \in {\bf R}^{n+2}$.

\begin{theorem} [Cartan]
\label{thm:1-Cartan-1} 
Let $M \subset S^{n+1} \subset {\bf R}^{n+2}$ be a connected isoparametric hypersurface with $g$ principal curvatures
$\lambda_i = \cot \theta_i$, $0 < \theta_1 < \cdots < \theta_g < \pi$, 
all having the same multiplicity $m$.
Then $M$ is an open subset of a level set of the restriction to $S^{n+1}$ of a homogeneous polynomial $F$ 
on ${\bf R}^{n+2}$ of degree $g$ satisfying the differential equations,
\begin{equation}
\label{eq:1-Muenzner-diff-eq-1-car}
\Delta_1 F = |{\mbox{\rm grad }}F|^2 = g^2 r^{2g-2} = g^2 (x_1^2 + x_2^2 + \cdots + x_{n+2}^2)^{{g-1}},
\end{equation}
\begin{equation}
\label{eq:1-Muenzner-diff-eq-2-car}
\Delta_2 F = \Delta F = 0 \ (F\ {\mbox{\it is\ harmonic}}),
\end{equation}
where $r = |x|$.
\end{theorem}

Cartan's Theorem \ref{thm:1-Cartan-1} was
a forerunner of M\"{u}nzner's general result (see Theorem \ref{thm:1-Muenzner-1} below)
that every isoparametric hypersurface with $g$ distinct principal curvatures
is algebraic (regardless of the multiplicities of the principal curvatures), and its
defining homogeneous polynomial $F$ of degree $g$ satisfies certain conditions on $\Delta_1 F$ and $\Delta_2 F$, 
which generalize the conditions that Cartan found in Theorem \ref{thm:1-Cartan-1}.

\subsection{Hypersurfaces With Three Principal Curvatures}
\label{cartan-case g=3} 

The classification of isoparametric hypersurfaces with $g=3$ distinct principal curvatures in \cite{Car3} is certainly
one of  the 
highlights of Cartan's work on isoparametric hypersurfaces.  In this paper, Cartan gives a complete classification 
of isoparametric hypersurfaces in $S^{n+1}$ with $g=3$ principal curvatures,
and he describes each example in great detail.

Let $M^n \subset S^{n+1} \subset {\bf R}^{n+2}$ be a connected, oriented hypersurface  with field of unit normals $\xi$
having $g=3$ distinct principal curvatures at every point.  The first step in Cartan's classification theorem is to prove the following theorem \cite[pp. 359--360]{Car3}.

\begin{theorem}
\label{thm:isop-hyp-g-3-multiplicities} 
Let $M^n \subset S^{n+1}$  be a connected isoparametric hypersurface with $g=3$ principal curvatures at each point.  Then all of the principal curvatures have the same multiplicity $m$.
\end{theorem}

Cartan proves this using the method of moving frames and conditions on the structural formulas that follow from the isoparametric condition and the assumption that $g=3$, and we refer the reader to Cartan's paper
\cite[pp. 359--360]{Car3} for the proof.

Next, since all of the multiplicities are equal, 
Cartan's Theorem \ref{thm:1-Cartan-1} \cite[pp. 364-367]{Car3}
above implies that $M^n$ is an open subset of a level set of the restriction to $S^{n+1}$ of a harmonic
homogeneous polynomial $F$ 
on ${\bf R}^{n+2}$ of degree $g=3$ satisfying the differential equations stated in Theorem \ref{thm:1-Cartan-1}.

Cartan's approach to classifying isoparametric hypersurfaces with
$g=3$ 
is to try to directly determine all homogeneous harmonic polynomials $F$ of degree $g=3$ that satisfy the differential equation from Theorem \ref{thm:1-Cartan-1},

\begin{equation}
\label{eq:1-Muenzner-diff-eq-1-a}
\Delta_1 F = |{\mbox{\rm grad }}F|^2 = g^2 r^{2g-2} = g^2 (x_1^2 + x_2^2 + \cdots + x_{n+2}^2)^{{g-1}}.
\end{equation}
After a long calculation, Cartan shows that the algebraic determination of such a harmonic polynomial requires the possibility of solving
the algebraic problem:\\

\noindent
{\bf Problem 1:} {\it Represent the product
\begin{displaymath}
(u_1^2 + u_2^2 + \cdots + u_m^2) (v_1^2 + v_2^2 + \cdots + v_m^2)
\end{displaymath}
of two sums of $m$ squares by the sum of the squares of $m$ bilinear combinations of the $u_i$ and the $v_j$}.\\

\noindent
By a theorem of Hurwitz \cite{Hurwitz} on normed linear algebras, this is only possible if $m = 1, 2, 4,$ or 8.

Cartan \cite[pp. 340--341]{Car3} relates the solution of Problem 1 above to the theory of Riemannian spaces admitting an isogonal absolute parallelism, and he uses results from a joint paper of himself and J. A. Schouten 
\cite{cartan-schouten} on that topic to determine the possibilities for $F$.

Ultimately, Cartan \cite{Car3} shows that the required harmonic homogeneous polynomial $F$ of degree 3 
on ${\bf R}^{3m+2}$ must be of the form
$F(x,y,X,Y,Z)$  given by

\begin{equation}
\label{eq:1-Cartan-poly-g-3}
x^3 - 3xy^2 + \frac{3}{2} x(X \overline{X} + Y \overline{Y} - 2 Z \overline{Z}) + \frac{3 \sqrt{3}}{2}
y (X\overline{X} - Y \overline{Y}) + \frac{3 \sqrt{3}}{2} (XYZ + \overline{Z} \overline{Y} \overline{X}).
\end{equation}
In this formula, $x$ and $y$ are real parameters, while $X,Y,Z$ are coordinates in the division algebra
${\bf F} = {\bf R}, {\bf C}, {\bf H}$ (quaternions), ${\bf O}$ (Cayley numbers),
for $m = 1,2,4,8$,
respectively.  

Note that the sum $XYZ + \overline{Z} \overline{Y} \overline{X}$ is twice the real part
of the product $XYZ$.  In the case of the Cayley numbers, multiplication is not associative, but the
real part of $XYZ$ is the same whether one interprets the product as $(XY)Z$ or $X(YZ)$. 

The isoparametric
hypersurfaces in the family are the level sets $M_t$ in $S^{3m+1}$ determined by the equation
$F = \cos 3t$, $0 < t < \pi/3$, where $F$ is the polynomial in
equation \eqref{eq:1-Cartan-poly-g-3}.  The focal submanifolds are obtained by taking $t=0$ and $t= \pi/3$.  These 
focal submanifolds are
a pair of antipodal standard embeddings of the projective plane
${\bf FP}^2$, for the appropriate division
algebra ${\bf F}$ listed above
(see, for example, \cite[pp. 74--78]{CR8}).  In the case of ${\bf F} = {\bf R}$,
these focal submanifolds are standard Veronese surfaces in $S^4 \subset {\bf R}^5$. 

For the cases
${\bf F} = {\bf R}, {\bf C}, {\bf H}$, Cartan gave a specific parametrization of the focal submanifold $M_0$
defined by the condition $F=1$ as follows:
\begin{eqnarray}
\label{eq:standard-embedding-focal-set}
X = \sqrt{3} v \overline{w}, \quad Y = \sqrt{3} w \overline{u}, \quad Z = \sqrt{3} u \overline{v},\\
x = \frac{\sqrt{3}}{2}(|u|^2 - |v|^2), \quad y = |w|^2 - \frac{|u|^2 + |v|^2}{2},\nonumber
\end{eqnarray}
where $u,v,w$ are in ${\bf F}$, and $|u|^2 + |v|^2 + |w|^2 = 1$.  This map is invariant under the
equivalence relation
\begin{displaymath}
(u,v,w) \sim (u\lambda,v\lambda,w\lambda), \quad \lambda \in {\bf F},\  |\lambda| = 1.
\end{displaymath}
Thus, it is well-defined on ${\bf FP}^2$, and it is easily shown to be injective on ${\bf FP}^2$. Therefore, it is
an embedding of ${\bf FP}^2$ into $S^{3m+1}$. 

 Cartan states that he does not know of an analogous representation of the focal variety in the case $m = 8$. 
 In that case, if  $u, v, w$ are taken to be Cayley numbers, then
the formulas \eqref{eq:standard-embedding-focal-set} are not well defined on the Cayley projective plane.
For example, the product $v \overline{w}$ is not preserved, in general, if we replace $v$ by  $v\lambda$, and $w$ by $w\lambda$, where $\lambda$ is a unit Cayley number.  

Note that even though there is no parametrization corresponding to \eqref{eq:standard-embedding-focal-set},
the Cayley projective plane can be described (see, for example,  Kuiper \cite{Ku3}, Freudenthal \cite{Freu})
as the submanifold
\begin{displaymath}
V = \{A \in M_{3 \times 3} ({\bf O}) \mid \overline{A}^T = A = A^2,\ {\rm trace}\ A = 1 \},
\end{displaymath}
where $M_{3 \times 3} ({\bf O})$ is the space of $3 \times 3$ matrices of Cayley numbers.  This submanifold
$V$ lies in a sphere $S^{25}$ in a $26$-dimensional real subspace of $M_{3 \times 3} ({\bf O})$.

In his paper, Cartan \cite[pp. 342--358]{Car3} writes a separate section for each of the cases $m = 1,2,4,8$, that is, ${\bf F} = {\bf R}, 
{\bf C}, {\bf H},{\bf O}$.  He gives extensive details in each case, describing many remarkable properties, especially in the case $m=8$. 
In particular, he shows that in each case the isoparametric hypersurfaces and the two focal
submanifolds are homogeneous, that is, they are orbits of points in $S^{n+1}$ under the action
of a closed subgroup of $SO(n+2)$.

Cartan's results show that up to congruence,
there is only one isoparametric family of hypersurfaces with $g=3$ principal curvatures for each value of $m$.
This classification is closely related to various characterizations of the standard embeddings of ${\bf FP}^2$.
(See Ewert \cite{Ewert}, Little \cite{Little}, and Knarr-Kramer \cite{K-K}.)

For alternative proofs of Cartan's 
classification of isoparametric hypersurfaces with $g=3$ principal curvatures,
see the papers of Knarr and Kramer \cite{K-K}, and Console and Olmos \cite{Console-Olmos-98}.
In a related paper, Sanchez \cite{Sanchez} studied Cartan's isoparametric hypersurfaces from an algebraic point of view. 
(See also the paper of Giunta and Sanchez \cite{Giunta-Sanchez-2014}.)

\subsection{Hypersurfaces With Four Principal Curvatures}
\label{cartan-case g=4} 

In his fourth paper on the subject, Cartan \cite{Car5} continues
his study of isoparametric hypersurfaces with the property that all the principal curvatures have the same multiplicity $m$.
In that paper, Cartan produces a family of 
isoparametric hypersurfaces with $g=4$ principal curvatures of 
multiplicity $m=1$ in $S^5$,
and a family with $g=4$ principal curvatures of multiplicity $m=2$ in $S^9$.  

In both cases, Cartan describes the geometry and topology of the hypersurfaces and the two focal submanifolds in great detail, pointing out several notable properties that they have.  
(See also the papers of Nomizu \cite{Nom3}--\cite{Nom4} and the book \cite[pp.155--159]{CR8} for more detail
on the example with $m=1$.)  Cartan also proves that the hypersurfaces and the focal submanifolds are homogeneous in both cases,  $m=1$ and $m=2$.

It is worth noting that Cartan's example with four principal curvatures of multiplicity $m=1$ is the lowest dimensional example of an isoparametric 
hypersurface of FKM-type constructed by Ferus, Karcher and M\"{u}nzner \cite {FKM} using representations of Clifford
algebras. Furthermore,
Cartan's example with four principal curvatures of multiplicity $m=2$ is not of FKM-type.  Its principal curvatures do not 
have the correct multiplicities for a hypersurface of FKM-type (see \cite {FKM}).

Cartan says that he can show by a very long calculation (which he does not give) that for an isoparametric hypersurface
with $g=4$ principal curvatures having the same multiplicity $m$, the only possibilities are $m=1$ and $m=2$.  This was shown to be true
later by Grove and Halperin \cite{GH}.
Furthermore, up to congruence, Cartan's examples
given in  \cite{Car5} are the only possibilities for  $m=1$ and $m=2$.  This was shown to be true 
by Takagi \cite{Takagi} in the case $m=1$, and by Ozeki and Takeuchi \cite{OT}--\cite{OT2} in the case $m=2$.

\subsection{Cartan's Questions}
\label{cartan-questions} 

Cartan's third paper \cite{Car4} is basically a survey of the results that he obtained in the other
three papers. As we have noted,
all of the examples that Cartan found
are homogeneous, each being an orbit of a point under an appropriate closed subgroup of $SO(n+2)$.  
Based on his results and the properties of his examples, Cartan asked the following three questions 
at the end of his survey paper \cite{Car4}. All of his questions were answered
in the 1970's, as we will describe below.

\begin{enumerate}
\item For each positive integer $g$, does there exist an isoparametric family with $g$ distinct principal
curvatures of the same multiplicity?\\
\item Does there exist an isoparametric family of hypersurfaces with more than three distinct principal curvatures such
that the principal curvatures do not all have the same multiplicity?\\
\item Does every isoparametric family of hypersurfaces admit a transitive group of isometries?
\end{enumerate}

In the early 1970's, Nomizu
\cite{Nom3}--\cite{Nom4} wrote two papers describing the highlights of Cartan's work.  He also generalized 
Cartan's example with four principal curvatures of multiplicity one to produce examples with four principal 
curvatures having multiplicities $m_1 = m_3 = m$, and $m_2 = m_4 =1$, for any positive integer $m$.  This answered
Cartan's Question 2 in the affirmative.  

In 1972, Takagi and Takahashi \cite{TT} gave a complete classification of all homogeneous 
isoparametric hypersurfaces
in $S^{n+1}$ based on the work of Hsiang and Lawson \cite{HsL}.  Takagi and Takahashi  \cite[p.480]{TT}
showed that each homogeneous
isoparametric hypersurface in $S^{n+1}$ is a principal orbit of the isotropy representation
of a Riemannian symmetric space of 
rank 2, and they gave a complete list of examples.  This list contains examples with $g=6$
principal curvatures as well as those with $g=1,2,3,4$ principal curvatures. In some cases with $g=4$, the principal
curvatures do not all have the same multiplicity, so this also provided an affirmative answer to 
Cartan's Question 2.  

At about the same time as the papers of Nomizu and Takagi-Takahashi, M\"{u}nzner \cite{Mu}--\cite{Mu2}
published two preprints that 
greatly extended Cartan's work and have served as the basis for much of the research in the field since that
time. The preprints were eventually published as journal articles \cite{Mu}--\cite{Mu2} in 1980--1981, and they are the basis of Section \ref{munzner-work} of this paper below.

One of M\"{u}nzner's primary results is that the number $g$ of distinct principal curvatures of an isoparametric hypersurface in a sphere must be equal to $1,2,3,4$ or 6,
and thus the answer to Cartan's Question 1 is negative.

Finally, the answer to Cartan's Question 3 is also negative, as was first shown by the construction of 
some families of inhomogeneous
isoparametric hypersurfaces with $g=4$ principal curvatures
by Ozeki and Takeuchi \cite{OT} in 1975.  Their construction was then generalized in 1981 to yield even more inhomogeneous examples by 
Ferus, Karcher and M\"{u}nzner \cite{FKM}.

\begin{remark}[Isoparametric submanifolds of codimension greater than one]
\label{rem:isop-subm-codim-greater-than-1}

\noindent
{\rm There is also an extensive theory of 
isoparametric submanifolds
of codimension greater than one in the sphere, due primarily to Carter and West  \cite{CW6}--\cite{CW8}, West \cite{West1},
Terng  \cite{Te1}--\cite{Te5}, and 
Hsiang, Palais and Terng \cite{HPT}.
(See also Harle \cite{Har} and 
Str\"{u}bing \cite{Str}.)

Terng \cite{Te1} formulated the definition as follows:
a connected, complete submanifold $V$ in a real space form $\widetilde{M}^{n+1}$ is said to be {\em isoparametric} if it has 
flat normal bundle
and if for any parallel section of the unit normal bundle $\eta:V \rightarrow B^n$, the principal curvatures of
$A_\eta$ are constant. 

After considerable development of the theory, Thorbergsson \cite{Th5}
showed that a compact, irreducible isoparametric submanifold $M$ substantially embedded in $S^{n+1}$ with codimension greater than one is homogeneous.
Thus, $M$ is a principal orbit 
of an isotropy representation of a symmetric space (also called $s$-representations), 
as in the codimension one case.  Orbits of isotropy representations of symmetric spaces are 
also known as {\em generalized flag manifolds} or {\em $R$-spaces}. (See Bott-Samelson \cite{BS} and Takeuchi-Kobayashi \cite{TK}).} 
\end{remark}

\section{M\"{u}nzner's Work}
\label{munzner-work} 

In this section, we cover the work of M\"{u}nzner in his two papers \cite{Mu}--\cite{Mu2}.  
We will give detailed proofs of several of the main results.
Our treatment is based 
primarily on M\"{u}nzner's first paper \cite{Mu} and on the 
book of Cecil and Ryan \cite[pp. 102--137]{CR8}.  Some passages in this section are taken directly from that book.

\subsection{Principal Curvatures of Parallel Hypersurfaces}
\label{sec:1.2}
We now begin our treatment of M\"{u}nzner's theory, which is contained in Subsections \ref{sec:1.2}--\ref{sec:1.6} of this paper. We start by deriving the well-known formulas for the principal curvatures of a parallel hypersurface to a given hypersurface.  We include this for the sake of completeness, since these formulas are important in the development of the theory.

For the following local calculations, we consider a connected, oriented isoparametric hypersurface 
$M \subset S^{n+1} \subset {\bf R}^{n+2}$ with field of unit normals $\xi$.  Assume that $M$ has $g$ distinct 
constant principal curvatures at each point, which we label by,
\begin{equation}
\label{eq:1-prin-curv-sph}
\lambda_i = \cot \theta_i, \ 0 < \theta_i < \pi, \ 1 \leq i \leq g,
\end{equation}
where the $\theta_i$ form an increasing sequence, and $\lambda_i$ has multiplicity $m_i$ on $M$.

We denote the corresponding principal distribution by,
\begin{equation}
\label{eq:1-prin-distr}
T_i (x) = \{X \in T_x M \mid AX = \lambda_i X \},
\end{equation}
where $A$ is the shape operator determined by the field of unit normals $\xi$. 

Using the Codazzi equations in the case of multiplicity $m_i > 1$, and by the theory of ordinary differential
equations in the case $m_i = 1$, one can show that each $T_i$ is integrable, i.e., it is
a foliation of $M$ with leaves of dimension $m_i$.
Furthermore, the leaves of the principal foliation $T_i$ corresponding to $\lambda_i$ are open subsets of 
$m_i$-dimensional metric spheres in $S^{n+1}$, and the space of leaves $M/T_i$ is an
$(n-m_i)$-dimensional manifold with the quotient topology (see \cite[pp. 18--32]{CR8} for proofs).

We consider the parallel hypersurface $f_t: M \rightarrow S^{n+1}$ defined by
\begin{equation}
\label{eq:1-parallel-hyp}
f_t (x) = \cos t \  x + \sin t \ \xi(x),
\end{equation}
that is, $f_t (x)$ is the point in $S^{n+1}$ at an oriented distance $t$ along the normal geodesic in $S^{n+1}$
to $M$ through the point $x$.  Note that $f_0$ is the original embedding $f$ defined by $f(x) = x$, and we suppress
the mention of $f$ in equation \eqref{eq:1-parallel-hyp}.

In the following calculations, we show that $f_t$ is an immersion at $x$ if and only if $\cot t$ is not a principal
curvature of $M$ at $x$.  In that case, we then find the principal curvatures of $f_t$ in terms of the principal
curvatures of the original embedding $f$.  

Let $X \in T_xM$.  Then differentiating equation (\ref{eq:1-parallel-hyp}) in the direction $X$, we get
\begin{displaymath}
(f_t)_* X = \cos t \  X + \sin t \ D_X \xi = \cos t \ X - \sin t \ AX = (\cos t \ I - \sin t \ A)X,
\end{displaymath}
where on the right side we are identifying $X$ with its Euclidean parallel translate at $f_t(x)$,
and $D$ is the Euclidean covariant derivative on ${\bf R}^{n+2}$.
If $X \in T_i (x)$,
this yields,
\begin{equation}
\label{eq:1-diff-parallel-hyp-2}
(f_t)_* X = (\cos t - \sin t \cot \theta_i)\  X= \frac{\sin(\theta_i - t)}{\sin \theta_i} X.
\end{equation}
Since $T_x M$ is 
the direct sum of the principal spaces $T_i (x)$, we see that $(f_t)_*$ is injective on $T_x M$, unless 
$t= \theta_i$ (mod $\pi$) for some $i$, that is, unless $f_t (x)$ is a focal point of $(M,x)$.  

Thus, we see that a parallel hypersurface
$f_t M$ is an immersed hypersurface if $t \neq \theta_i$ (mod $\pi$) for any $i$.  In that case, we want to
find the principal curvatures of $f_t M$.  

\begin{theorem}
\label{thm:1-prin-curv-par-hyp} 
Let $M \subset S^{n+1}$ be an oriented isoparametric  hypersurface and let $\lambda = \cot \theta$ be a principal curvature  of multiplicity $m$ on $M$. Suppose that $f_t$ is an immersion of $M$, and let  $x$ be a point in $M$.
Then the parallel hypersurface $f_t M$ has a principal curvature $\tilde{\lambda} = \cot (\theta - t)$
at $f_t(x)$ having the same multiplicity $m$ and (up to parallel translation in ${\bf R}^{n+2}$) the same principal space $T_{\lambda} (x)$ as $\lambda$ at $x$.
\end{theorem}

\begin{proof}
 Let $\xi$ be the field of unit normals on $M$, and denote the value $\xi$ at a point $y \in M$ by $\xi_y$. Then 
we can easily compute that the vector
\begin{equation}
\label{eq:1-tilde-xi}
\tilde{\xi}_y = - \sin t\  y + \cos t\  \xi_y,
\end{equation}
when translated to $\tilde{y} = f_t (y)$, is a unit normal to the hypersurface $f_t M$ at the point $\tilde{y}$.  We want to find
the shape operator $A_t$ determined by this field of unit normals $\tilde{\xi}$ on $f_tM$.  Let $X \in T_{\lambda} (x)$.
Since $\langle X,\xi \rangle = 0$,
we have 
\begin{equation}
\label{eq:1-cov-deriv}
D_X \xi = \widetilde{\nabla}_X \xi - \langle X,\xi \rangle\  x = \widetilde{\nabla}_X \xi = - AX = -\lambda X = - \cot \theta\  X,
\end{equation}
where $D$ is the Euclidean covariant derivative on ${\bf R}^{n+2}$, and $\widetilde{\nabla}$ is the induced
Levi-Civita connection on $S^{n+1}$. By definition, the shape operator $A_t$ is given by
\begin{equation}
\label{eq:1-shape-op-A-t}
(f_t)_* (A_t X) = - \widetilde{\nabla}_{(f_t)_*X} \tilde{\xi} = - D_{(f_t)_*X} \tilde{\xi},
\end{equation}
since $\langle (f_t)_*X, \tilde{\xi} \rangle = 0$.
To compute this, let $x_u$ be a curve in $M$ with initial point
$x_0 = x$ and initial tangent vector $\overrightarrow{x_0} = X$.  Then we have using equation (\ref{eq:1-cov-deriv}),

\begin{eqnarray}
\label{eq:1-differential-tilde-xi}
D_{(f_t)_*X} \tilde{\xi} = \frac{d}{du} (\tilde{\xi}_{x_u})|_{u=0} & = & - \sin t\  X - \cos t \cot \theta\  X  \\
& = & \frac{- \cos (\theta - t)}{\sin \theta} X,\nonumber
\end{eqnarray}
where we again identify $X$ on the right with its Euclidean parallel translate at $f_t(x)$.  If we compare this with
equation (\ref{eq:1-diff-parallel-hyp-2}) and use equation (\ref{eq:1-shape-op-A-t}), we see that 
\begin{equation}
\label{eq:1-A-t-X}
A_t X = \cot (\theta - t) X,
\end{equation}
for $X \in T_{\lambda} (x)$, and so $\tilde{\lambda} = \cot (\theta - t)$ is a principal curvature of $A_t$
with the same principal space $T_{\lambda} (x)$ and same multiplicity $m$ as $\lambda$ at $x$.
\end{proof}

As a consequence of Theorem \ref{thm:1-prin-curv-par-hyp}, we get the following corollary regarding the
principal curvatures of a family of parallel isoparametric hypersurfaces.

\begin{corollary}
\label{cor:1-prin-curv-isop-hyp} 
Let $M \subset S^{n+1}$ be a connected isoparametric hypersurface having $g$ distinct principal curvatures
$\lambda_i = \cot \theta_i$, $1 \leq i \leq g$, with respective multiplicities $m_i$.  If $t$ is any real
number not congruent to any $\theta_i$ (mod $\pi$), then the map $f_t$ immerses $M$ as an isoparametric hypersurface
with principal curvatures $\tilde{\lambda_i} =\cot (\theta_i - t)$, $1 \leq i \leq g$, with the same multiplicities $m_i$.
Furthermore, for each $i$, the principal foliation corresponding to $\tilde{\lambda}_i$ is the same as
the principal foliation $T_i$ corresponding to $\lambda_i$ on $M$. 
\end{corollary}

\begin{remark}
{\rm It follows from M\"{u}nzner's theory that if $M$ is an isoparametric hypersurface embedded in $S^{n+1}$, then each parallel
isoparametric hypersurface $f_tM$ is also embedded in $S^{n+1}$ and not just immersed. This is because $M$ and its
parallel hypersurfaces are level sets of the restriction to $S^{n+1}$ of a certain polynomial function on
${\bf R}^{n+2}$, as will be discussed later
in Theorem \ref{thm:1-Muenzner-1}.}
\end{remark}

\subsection{Focal Submanifolds}
\label{sec:1.3}
The geometry of the focal submanifolds is a crucial element in the theory of isoparametric hypersurfaces.
In this subsection, we obtain some important basic results about isoparametric hypersurfaces and their focal submanifolds
due to M\"{u}nzner \cite{Mu}.

We use the notation introduced in Subsection \ref{sec:1.2}.   So we let
$M^n \subset S^{n+1}$ be a connected, oriented isoparametric hypersurface with field of unit normals $\xi$
having $g$ distinct constant principal
curvatures,
\begin{equation}
\label{eq:1-prin-curv-isop-hyp}
\lambda_i = \cot \theta_i, \ 0 < \theta_i < \pi, \ 1 \leq i \leq g,
\end{equation}
where the $\theta_i$ form an increasing sequence. Denote the multiplicity of $\lambda_i$ by $m_i$.

As noted earlier, the leaves of the principal foliation $T_i$ corresponding to $\lambda_i$ are open subsets of 
$m_i$-dimensional metric spheres in $S^{n+1}$, and the space of leaves $M/T_i$ is an
$(n-m_i)$-dimensional manifold with the quotient topology (see \cite[pp. 18--32]{CR8} for proofs).

We consider the map 
$f_t$ defined as in equation \eqref{eq:1-parallel-hyp} by
\begin{displaymath}
f_t (x) = \cos t \  x + \sin t \ \xi(x),
\end{displaymath}
and we use the formulas developed in Subsection \ref{sec:1.2} for making computations.

If $t = \theta_i$, then
the map $f_t$ has constant rank $n-m_i$ on $M$, and  it factors through an
immersion $\psi_i :M/T_i \rightarrow S^{n+1}$, defined on the space of leaves $M/T_i$,
i.e., $f_t = \psi_i \circ \pi$, where $\pi$ is the projection from $M$ 
onto $M/T_i$. 
Thus, $f_t$ is a focal map, and we denote image of $\psi_i$
by $V_i$.

We now want to find the eigenvalues of the shape operators 
of this focal submanifold $V_i$.  This is similar to the calculation
for parallel hypersurfaces in Subsection \ref{sec:1.2}, although we must make some adjustments because
$V_i$ has codimension greater than one.

Let $x \in M$.  Then we have the orthogonal decomposition of the tangent space $T_xM$ as
\begin{equation}
\label{eq:orthog-decomp}
T_xM = T_i(x) \oplus T_i^\perp (x),
\end{equation}
where $T_i^\perp (x)$ is the direct sum of the spaces $T_j(x)$, for $j \neq i$.  By equation 
(\ref{eq:1-diff-parallel-hyp-2}) in Subsection \ref{sec:1.2}, the map
$(f_t)_* = 0$ on $T_i(x)$, and $(f_t)_*$ is injective on $T_i^\perp (x)$, for $t=\theta_i$.
We consider a map $h:M \rightarrow S^{n+1}$ given by
\begin{equation}
\label{eq:1-defn-h}
h(x) = - \sin t\  x + \cos t\  \xi (x).
\end{equation}
This is basically the same function that we considered in equation (\ref{eq:1-tilde-xi}) for the field $\tilde{\xi}$
of unit normals to $f_tM$, in the case where $f_t$ is an immersion.  

In the case $t=\theta_i$,
we see that the inner product
$\langle f_t(x), h(x) \rangle = 0$, and so the vector $h(x)$ is tangent to the sphere $S^{n+1}$
at the point $p = f_t (x)$.  Furthermore, $\langle h(x), X \rangle = 0$ for all $X \in T_i^\perp (x)$, and so by
equation (\ref{eq:1-diff-parallel-hyp-2}), the vector $h(x)$ is normal to the focal submanifold $V_i$ at the point $p$.

We can use the map $h$ to find the shape operator determined by a normal
vector to $V_i$ as follows.  Let $p$ be an arbitrary point in $V_i$.  Then the set $C = f_t^{-1} (p)$ is an open subset
of an $m_i$-sphere in $S^{n+1}$.  
For each $x \in C$, the vector $h(x)$ is a unit normal to the focal submanifold $V_i$ at $p$.  Thus,
the restriction of $h$ to $C$ is a map from $C$ into the $m_i$-sphere $S_p^\perp V_i$ of unit normal vectors
to $V_i$ at $p$. At a point $x \in C$, the tangent space $T_xC$ equals $T_i (x)$.
Since $t = \theta_i$, we compute for $x \in C$ and a nonzero vector $X \in T_i(x)$ that
\begin{eqnarray}
\label{eq:1-diff-h}
h_* (X)& = &- \sin t X + \cos t (-AX) = - \sin \theta_i X + \cos \theta_i (- \cot \theta_i)X\\
& = & \frac{-1}{\sin \theta_i} X \neq 0.\nonumber
\end{eqnarray}
Thus, $h_*$ has full rank $m_i$ on $C$, and so $h$ is a local diffeomorphism of open subsets of $m_i$-spheres.
This enables us to prove the following important result due to M\"{u}nzner \cite{Mu}.

\begin{theorem}
\label{thm:1-prin-curv-focal-submanifold}
Let $M \subset S^{n+1}$ be a connected isoparametric hypersurface, and let $V_i = f_tM$ for $t=\theta_i$
be a focal submanifold of $M$. Let $\eta$ be a unit normal
vector to $V_i$ at a point $p \in V_i$, and suppose that $\eta = h(x)$ for some $x \in f_t^{-1} (p)$.  Then the
shape operator $A_\eta$ of $V_i$ is given in terms of its principal vectors by
\begin{equation}
\label{eq:1-prin-curv-focal-submanifold}
A_\eta X = \cot (\theta_j - \theta_i)X, \ {\rm for}\ X \in T_j (x), \ j \neq i.
\end{equation}
(As before we are identifying $T_j(x)$ with its Euclidean parallel translate at $p$.)
\end{theorem}
\begin{proof}
Let $\eta = h(x)$ for some $x \in C = f_t^{-1} (p)$ for $t = \theta_i$.  The same calculation used in proving
Theorem \ref{thm:1-prin-curv-par-hyp} is valid here, and it leads to equation (\ref{eq:1-A-t-X}), which we write as
in equation \eqref{eq:1-prin-curv-focal-submanifold},
\begin{displaymath}
A_\eta X = \cot (\theta_j - \theta_i)X, \ {\rm for}\ X \in T_j (x), \ j \neq i.
\end{displaymath}
\end{proof}

\begin{corollary}
\label{cor:1-prin-curv-focal-submanifold}
Let $M \subset S^{n+1}$ be a connected isoparametric hypersurface, and let $V_i = f_tM$, for $t=\theta_i$,
be a focal submanifold of $M$. Then for every unit normal vector $\eta$ at every point $p \in V_i$, the
shape operator $A_\eta$ has eigenvalues $\cot (\theta_j - \theta_i)$ with multiplicities $m_j$,
for $j \neq i$, $1 \leq j \leq g$.
\end{corollary}

\begin{proof}
By Theorem \ref{thm:1-prin-curv-focal-submanifold}, the corollary holds on the open subset $h(C)$ of the 
$m_i$-sphere $S_p^\perp V_i$ of unit normal vectors to $V_i$ at $p$.  Consider the characteristic polynomial
$P_u (\eta) = \det (A_\eta - u I)$ as a function of $\eta$ on the normal space $T_p^\perp V_i$.  
Since $A_\eta$ is linear in $\eta$, we have for each fixed $u \in {\bf R}$ that the function $P_u (\eta)$ is a polynomial
of degree $n-m_i$ on the vector space $T_p^\perp V_i$. 
Thus, the restriction of $P_u (\eta)$ to the
sphere $S_p^\perp V_i$ is an analytic function of $\eta$.  Then since $P_u (\eta)$ is constant on the open subset
$h(C)$ of $S_p^\perp V_i$, it is constant on all of $S_p^\perp V_i$.
\end{proof}

\subsubsection{Minimality of the focal submanifolds}
\label{minimality-focal-submanifolds}

M\"{u}nzner also obtained the following consequence of Corollary \ref{cor:1-prin-curv-focal-submanifold}.  This result was
obtained independently with a different proof by Nomizu \cite{Nom3}.

\begin{corollary}
\label{cor:1-focal-sub-minimal}
Let $M \subset S^{n+1}$ be a connected isoparametric hypersurface.  Then each focal submanifold $V_i$ of $M$ is a
minimal submanifold in $S^{n+1}$.
\end{corollary}

\begin{proof}
Let $\eta$ be a unit normal vector to a focal submanifold $V_i$ of $M$.  Then $-\eta$ is also a unit normal vector to
$V_i$.  By Corollary \ref{cor:1-prin-curv-focal-submanifold}, the shape operators $A_\eta$ and $A_{-\eta}$ have the
same eigenvalues with the same multiplicities.  So 
\begin{displaymath}
{\rm trace}\  A_{-\eta}\  =\  {\rm trace}\  A_\eta.  
\end{displaymath}
On the other hand,
trace $A_{-\eta} = -$ trace $A_{\eta}$, since $A_{-\eta} = - A_{\eta}$.  Thus, we have ${\rm trace}\  A_\eta = -\  {\rm trace}\  A_\eta$,
and so ${\rm trace}\  A_\eta = 0$.  Since this
is true for all unit normal vectors $\eta$, we conclude that $V_i$ is a minimal submanifold in $S^{n+1}$.
\end{proof}

As a consequence of Theorem \ref{thm:1-prin-curv-focal-submanifold}, we can give a proof of Cartan's formula 
(see equation \eqref{eq:1-car-id-2} below)
for isoparametric hypersurfaces in $S^{n+1}$.  The proof shows that Cartan's formula is essentially equivalent to the minimality of the focal submanifolds.
\begin{corollary}
\label{cor:Cartan's identity}
(Cartan's formula) Let $M \subset S^{n+1}$ be a connected isoparametric hypersurface with $g$ principal curvatures
\begin{displaymath}
\lambda_i = \cot \theta_i, \quad 0 < \theta_1 < \cdots < \theta_g < \pi,
\end{displaymath}
with respective multiplicities $m_i$. 
Then for each $i,\  1 \leq i \leq g$, Cartan's formula holds, that is,
\begin{equation}
\label{eq:1-car-id-2}
\sum_{j\neq i} m_j \frac{1 + \lambda_i \lambda_j}{\lambda_i - \lambda_j} = 0.
\end{equation}
\end{corollary}

\begin{proof}
We will show that for each $i$ and for any unit normal $\eta$ to the focal submanifold $V_i$, the
left side of equation (\ref{eq:1-car-id-2}) equals trace $A_\eta$, which equals zero by
Corollary \ref{cor:1-focal-sub-minimal}.  Represent the principal curvatures of $M$ as
$\lambda_ i = \cot \theta_i$, $0 < \theta_1 < \cdots < \theta_g < \pi$, with respective multiplicities $m_i$.
Let $V_i$ be the focal submanifold corresponding to $\lambda_i$.  By
Corollary \ref{cor:1-prin-curv-focal-submanifold}, we have

\begin{eqnarray}
\label{eq:1-trace-shape-op}
0 = {\rm trace}\  A_\eta& = & \sum_{j \neq i} m_j \cot (\theta_j - \theta_i) = 
\sum_{j\neq i} m_j \frac{1 + \cot \theta_i \cot \theta_j}{\cot \theta_i - \cot \theta_j}  \\
& = & \sum_{j\neq i} m_j \frac{1 + \lambda_i \lambda_j}{\lambda_i - \lambda_j}.\nonumber 
\end{eqnarray}
\end{proof}

\subsubsection{Formula for the principal curvatures of $M$}
\label{formula-p-c}

M\"{u}nzner \cite[p. 61]{Mu} showed that the principal curvatures of an
isoparametric hypersurface in $S^{n+1}$
have a very specific form, given in equation \eqref{eq:1-prin-curv-formula} in
Theorem \ref{thm:1-prin-curv-isop-hyp} below. 
Cartan  \cite[pp. 186--187]{Car2} obtained this same formula 
in the case where the principal curvatures all have the same multiplicity (see 
Theorem \ref{thm:1-prin-curv-isop-hyp-cartan}), 
but he did not have the result in the general case.
  
Moreover, M\"{u}nzner's Theorem \ref{thm:1-prin-curv-isop-hyp} also contains the result that in the general case,
the multiplicities of the principal curvatures must satisfy the relation $m_i = m_{i+2}$ (subscripts mod $g$). 

Theorem \ref{thm:1-prin-curv-isop-hyp} is a key 
element in M\"{u}nzner's extension to the general case
of Cartan's Theorem \ref{thm:1-Cartan-1} concerning the algebraic nature of
isoparametric hypersurfaces (see
Theorem \ref{thm:1-Muenzner-1} below).  Our treatment of the proof of
M\"{u}nzner's Theorem \ref{thm:1-prin-curv-isop-hyp} is taken from \cite[pp. 108--110]{CR8}.

\begin{theorem} [M\"{u}nzner]
\label{thm:1-prin-curv-isop-hyp}
Let $M \subset S^{n+1}$ be a connected isoparametric hypersurface with $g$ principal curvatures
$\lambda_i = \cot \theta_i$, $0 < \theta_1 < \cdots < \theta_g < \pi$, with respective multiplicities $m_i$. Then
\begin{equation}
\label{eq:1-prin-curv-formula}
\theta_i = \theta_1 + (i-1) \frac{\pi}{g} , \quad 1 \leq i \leq g,
\end{equation}
and the multiplicities satisfy $m_i = m_{i+2}$ (subscripts mod $g$).  For any point $x \in M$, there
are $2g$ focal points of $(M,x)$ along the normal geodesic to $M$ through $x$, and they are evenly
distributed at intervals of length $\pi/g$.
\end{theorem}

\begin{proof}
If $g=1$, then the theorem is trivially true, so we now consider $g = 2$.
Let $V_1$ be the focal submanifold determined by the map $f_t$ for $t = \theta_1$.
By Corollary \ref{cor:1-prin-curv-focal-submanifold}, the eigenvalue $\cot (\theta_2 - \theta_1)$
of the shape operator $A_\eta$ is the same for every choice of unit normal $\eta$ at every 
point $p \in V_1$.  Since $A_{-\eta} = - A_\eta$, this says that 
\begin{displaymath}
\cot (\theta_2 - \theta_1) = - \cot (\theta_2 - \theta_1).
\end{displaymath}
Thus, $\cot (\theta_2 - \theta_1) = 0$, so $\theta_2 - \theta_1 = \pi/2$ as desired.  In the case $g=2$, there is no restriction on the multiplicities.

Next we consider the case $g \geq 3$.
For a fixed value of $i$, $1 \leq i \leq g$, let $V_i$ be the focal submanifold determined by the map $f_t$ for 
$t = \theta_i$.  By Corollary \ref{cor:1-prin-curv-focal-submanifold}, the set
\begin{displaymath}
\{\cot (\theta_j - \theta_i) \mid j \neq i\}
\end{displaymath}
of eigenvalues of the shape operator $A_\eta$ is the same for every choice of unit normal $\eta$ at every 
point $p \in V_i$. Since $A_{-\eta} = - A_\eta$, this says that the two sets
\begin{displaymath}
\{\cot (\theta_j - \theta_i) \mid j \neq i\}\ {\rm and}\ \{- \cot (\theta_j - \theta_i) \mid j \neq i\}
\end{displaymath}
are the same.
In the case $2 \leq i \leq g-1$, the largest eigenvalue of $A_\eta$ is $\cot (\theta_{i+1} - \theta_i)$ with multiplicity $m_{i+1}$, while the largest eigenvalue of $A_{-\eta}$ is
$\cot (\theta_i - \theta_{i-1})$ with multiplicity $m_{i-1}$.  Since these two largest eigenvalues and their respective multiplicities are equal, we conclude that
\begin{equation}
\label{eq:1-difference-theta-i}
\theta_{i+1} - \theta_i = \theta_i - \theta_{i-1}, \quad m_{i+1} = m_{i-1}, \quad 2 \leq i \leq g-1.
\end{equation}
If $i=1$, the largest eigenvalue of $A_\eta$ is $\cot (\theta_2 - \theta_1)$ with multiplicity $m_2$, and the largest
eigenvalue of $A_{-\eta}$ is
\begin{displaymath}
\cot (\theta_1 - \theta_g) = \cot (\theta_1 - (\theta_g - \pi)),
\end{displaymath}
with multiplicity $m_g$, and we have
\begin{equation}
\label{eq:1-difference-theta-1}
\theta_2 - \theta_1 = \theta_1 - (\theta_g - \pi), \quad m_2 = m_g.
\end{equation}
If we let $\theta_2 - \theta_1 = \delta$, then equation (\ref{eq:1-difference-theta-i}) implies that 
$\theta_g - \theta_1 = (g-1) \delta$, while equation (\ref{eq:1-difference-theta-1}) implies that
$\theta_g - \theta_1 = \pi - \delta$.  Combining these two equations, we get that $g \delta = \pi$, and thus
$\delta = \pi / g$.  From this we get the formula in equation (\ref{eq:1-prin-curv-formula}) for $\theta_i$.
The formula for the multiplicities in the theorem follows from
equations (\ref{eq:1-difference-theta-i}) and (\ref{eq:1-difference-theta-1}).  

If $x$ is any point of $M$, then each principal curvature $\cot \theta_i$
of $M$ gives rise to a pair of antipodal focal points along
the normal geodesic to $M$ through $x$.  Thus, there are $2g$ focal points of $(M,x)$ along this normal geodesic,
and they are evenly distributed at intervals of length $\pi/g$ by equation (\ref{eq:1-prin-curv-formula}).
\end{proof}

\begin{remark}[Isoparametric submanifolds and their Coxeter groups]
\label{rem:Coxeter-group}

\noindent
{\rm It follows from Theorem \ref{thm:1-prin-curv-isop-hyp} that the set of focal points along a normal circle to $M \subset S^{n+1}$ is invariant under the 
dihedral group $D_g$
of order $2g$ that acts on the normal circle and is generated by reflections in the focal points.
This is a fundamental idea that generalizes to isoparametric submanifolds of higher codimension in the sphere
(see Remark \ref{rem:isop-subm-codim-greater-than-1}).
Specifically, for an isoparametric submanifold $M^n$ of codimension $k > 1$ in $S^{n+1}$,
Carter and West \cite{CW6} (in the case $k=2$) and Terng \cite{Te1} for arbitrary $k > 1$
found a Coxeter group (finite group generated by reflections) that acts in a way similar to this dihedral group in the codimension one case.
This Coxeter group is important in the overall development of the theory in the case 
of isoparametric submanifolds of higher codimension (see Terng \cite{Te1} and Thorbergsson \cite{Th6}).}
\end{remark}

Since the multiplicities in Theorem \ref{thm:1-prin-curv-isop-hyp}
satisfy $m_i = m_{i+2}$ (subscripts mod $g$), we have the following immediate corollary.

\begin{corollary}
\label{cor:1-multiplicities-isop-hyp}
Let $M \subset S^{n+1}$ be a connected isoparametric hypersurface with $g$ distinct principal curvatures.  If $g$
is odd, then all of the principal curvatures have the same multiplicity.  If $g$ is even, then there are at most two
distinct multiplicities.
\end{corollary}

\subsection{Cartan-M\"{u}nzner Polynomials}
\label{c-m-polynomials}

Our goal in this subsection is to outline the proof of M\"{u}nzner's generalization of Cartan's
Theorem \ref{thm:1-Cartan-1} concerning the algebraic nature of isoparametric hypersurfaces in $S^{n+1}$.
Here we follow the treatment in the book \cite[pp. 111--130]{CR8} closely.

As noted at the beginning of Section \ref{Cartan-work},
the original definition of an isoparametric family of hypersurfaces in a real space
form $\widetilde{M}^{n+1}$ was formulated in terms of the level sets of an isoparametric function, 
where a smooth
function $V:\widetilde{M}^{n+1} \rightarrow {\bf R}$ is called isoparametric if both of the classical Beltrami
differential parameters,
\begin{equation}
\label{eq:1-beltrami-2}
\Delta_1 V = |{\rm grad}\  V|^2, \quad \Delta_2 V = \Delta V\ ({\rm Laplacian}\ V),
\end{equation}
are functions of $V$ itself, and are therefore constant on the level sets of $V$ in $\widetilde{M}^{n+1}$.

In the case of an isoparametric hypersurface $M$ in the unit sphere $S^{n+1}$ in ${\bf R}^{n+2}$ with $g$ distinct principal curvatures,  M\"{u}nzner \cite{Mu} showed that the corresponding
isoparametric function $V:S^{n+1} \rightarrow {\bf R}$ is the restriction to $S^{n+1}$ of a homogeneous
polynomial $F:{\bf R}^{n+2} \rightarrow {\bf R}$ of degree $g$ satisfying certain differential equations.  Thus, it is useful
to relate the Beltrami differential parameters of a homogeneous function $F$ on ${\bf R}^{n+2}$ to those of its
restriction $V$ to $S^{n+1}$.

Recall that a
function $F:{\bf R}^{n+2} \rightarrow {\bf R}$ is {\em homogeneous of degree}
$g$ if $F(tx) = t^g F(x)$, for all
$t \in {\bf R}$ and $x \in {\bf R}^{n+2}$.  By Euler's Theorem, we know that for any $x \in {\bf R}^{n+2}$,
\begin{equation}
\label{eq:1-homogeneous}
\langle {\rm grad}^E F, x \rangle = g F(x).
\end{equation}
Here the superscript $E$ is used to denote the Euclidean gradient of $F$.  The gradient
of the restriction $V$ of $F$ to $S^{n+1}$ 
will be denoted by ${\rm grad}^S F$.  
Similarly, we will denote the respective Laplacians by
$\Delta^E F$ and $\Delta^S F$.  

The following theorem relates the various differential operators
for a homogeneous function $F$ of degree $g$, and it is useful in
proving Theorem \ref{thm:1-Muenzner-1}. (See \cite[pp. 112--113]{CR8} for a proof.)

\begin{theorem}
\label{thm:1-beltrami-homogeneous} 
Let $F:{\bf R}^{n+2} \rightarrow {\bf R}$ be a homogeneous function of degree $g$.  Then
\begin{enumerate}
\item[${\rm(a)}$] $|{\rm grad}^S F|^2 = |{\rm grad}^E F|^2 - g^2 F^2$,
\item[${\rm(b)}$] $\Delta^S F = \Delta^E F - g (g-1) F - g (n+1) F$.
\end{enumerate}
\end{theorem}

M\"{u}nzner's generalization of Cartan's Theorem \ref{thm:1-Cartan-1} is the following.

\begin{theorem}[M\"{u}nzner]
\label{thm:1-Muenzner-1} 
Let $M \subset S^{n+1} \subset {\bf R}^{n+2}$ be a connected isoparametric hypersurface with $g$ 
principal curvatures
$\lambda_i = \cot \theta_i$, $0 < \theta_1 < \cdots < \theta_g < \pi$, with respective multiplicities $m_i$.
Then $M$ is an open subset of a level set of the restriction to $S^{n+1}$ of a homogeneous polynomial $F$ 
on ${\bf R}^{n+2}$ of degree $g$ satisfying the differential equations,
\begin{equation}
\label{eq:1-Muenzner-diff-eq-1}
|{\rm grad}^E F|^2 = g^2 r^{2g-2},
\end{equation}
\begin{equation}
\label{eq:1-Muenzner-diff-eq-2}
\Delta^E F = c r^{g-2},
\end{equation}
where $r = |x|$, and $c = g^2 (m_2 - m_1)/2$.
\end{theorem}

\begin{remark}
\label{recall-multiplicities}
{\rm Recall from Corollary \ref{cor:1-multiplicities-isop-hyp} that there are at most two distinct multiplicities $m_1, m_2$,
and the multiplicities satisfy $m_{i+2} = m_i$ (subscripts mod $g$).}
\end{remark}

M\"{u}nzner \cite[p. 65]{Mu}  called $F$ the {\em Cartan polynomial} of $M$, and now $F$ is usually referred to as the
{\em Cartan-M\"{u}nzner polynomial} of $M$.  Equations (\ref{eq:1-Muenzner-diff-eq-1})--(\ref{eq:1-Muenzner-diff-eq-2})
are called the {\em Cartan-M\"{u}nzner differential equations}.
By Theorem \ref{thm:1-beltrami-homogeneous} the
restriction $V$ of $F$ to $S^{n+1}$ satisfies the differential equations,

\begin{equation}
\label{eq:1-Muenzner-diff-eq-S-1}
|{\rm grad}^S V|^2 =  g^2 (1 - V^2),
\end{equation}
\begin{equation}
\label{eq:1-Muenzner-diff-eq-S-2}
\Delta^S V  = c - g(n+g) V,
\end{equation}
where $c = g^2 (m_2 - m_1)/2$. Thus $V$ is an isoparametric function in the sense of Cartan, since both
$|{\rm grad}^S V|^2$ and $\Delta^S V$ are functions of $V$ itself.  (Here we are using the superscript $S$ in the notations
${\rm grad}^S V$ and $\Delta^S V$ to emphasize that $V$ is a function defined on $S^{n+1}$.)

We now describe M\"{u}nzner's \cite[pp. 62--65]{Mu}
construction of this polynomial $F$ in detail.  This is basically the same approach used by Cartan \cite{Car3} in proving Theorem \ref{thm:1-Cartan-1} under the more restrictive assumption that all of the principal
curvatures have the same multiplicity.  

Note that Cartan \cite[pp. 364--365]{Car3} proved that if the
principal curvatures all have the same multiplicity $m$, then the polynomial $F$ must be harmonic.  This agrees with 
M\"{u}nzner's condition \eqref{eq:1-Muenzner-diff-eq-2} in the case where all of the multiplicities are equal,
since then the number $c = g^2 (m_2 - m_1)/2$ in equation \eqref{eq:1-Muenzner-diff-eq-2} equals zero,
and so $F$ is harmonic.

Let $M \subset S^{n+1}$ be a connected,
oriented isoparametric hypersurface with $g$ distinct principal curvatures
$\lambda_ i = \cot \theta_i$, $0 < \theta_1 < \cdots < \theta_g < \pi$, with respective multiplicities $m_i$.
The normal bundle $NM$ of $M$ in $S^{n+1}$ is trivial and is therefore diffeomorphic to $M \times {\bf R}$.  Thus we
can consider the normal exponential map $E:M \times {\bf R} \rightarrow S^{n+1}$ defined by

\begin{equation}
\label{eq:1-normal-exp-map}
E(x,t) = f_t (x) = \cos t \ x + \sin t \ \xi (x),
\end{equation}
where $\xi$ is the field of unit normals to $M$ in $S^{n+1}$. 

By the standard calculation for the location of focal points (see, for example, \cite[pp. 11--14]{CR8}), we know that the differential of $E$ has rank $n+1$ at 
$(x,t) \in M \times {\bf R}$, unless $\cot t$ is a principal curvature of $M$ at $x$.  Thus, for any noncritical
point $(x,t)$ of $E$, there is an open neighborhood $U$ of $(x,t)$ in $M \times {\bf R}$ on which $E$ restricts
to a diffeomorphism onto an open subset $\widetilde{U} = E(U)$ in $S^{n+1}$.  We define a function 
$\tau: \widetilde{U} \rightarrow {\bf R}$ by
\begin{equation}
\label{eq:1-defn-tau}
\tau (p) = \theta_1 - \pi_2(E^{-1} (p)),
\end{equation}
where $\pi_2$ is projection onto the second coordinate.  That is, if $p = E(x,t)$, then
\begin{equation}
\label{eq:1-tau}
\tau (p) = \theta_1 - t.
\end{equation}
Then we define a function $V:\widetilde{U} \rightarrow {\bf R}$ by,
\begin{equation}
\label{eq:1-V}
V (p) = \cos (g \tau (p)).
\end{equation}
Clearly, $\tau$ and $V$ are constant on each parallel hypersurface $M_t = f_t (M)$ in $\widetilde{U}$.

The number $\tau (p)$ is the oriented distance from $p$ to the first focal point along the normal geodesic to the
parallel hypersurface $M_t$ through $p$.  Thus, if we begin the construction with a parallel hypersurface 
near $M$ rather than with $M$ itself, we get the same functions $\tau$ and $V$.  If we begin the construction with 
the opposite field of unit normals $-\xi$ instead of $\xi$, then we obtain the function $-V$ instead of $V$.

We next extend $V$ to a homogeneous function of degree $g$ on the cone in ${\bf R}^{n+2}$ over $\widetilde{U}$ by the formula

\begin{equation}
\label{eq:1-def-F}
F(rp) = r^g \cos (g(\tau (p)),\ p \in \widetilde{U}, \ r>0.
\end{equation}

The first step in the proof of Theorem \ref{thm:1-Muenzner-1} is to show that the function $F$ in equation (\ref{eq:1-def-F})
satisfies the Cartan-M\"{u}nzner
differential equations (\ref{eq:1-Muenzner-diff-eq-1})--(\ref{eq:1-Muenzner-diff-eq-2}).
One then completes the proof of Theorem \ref{thm:1-Muenzner-1} by showing that $F$ is the restriction to the cone
over $\widetilde{U}$ of a homogeneous polynomial of degree $g$.  

These two steps involve lengthy calculations based on
the formula for the principal curvatures of an isoparametric
hypersurface given in Theorem \ref{thm:1-prin-curv-isop-hyp}.  We refer the reader to
 \cite[pp. 62--67]{Mu} or
\cite[pp. 115--125]{CR8} for the details of the proof, which we will omit here.

We now want to make a few remarks concerning some important consequences of Theorem \ref{thm:1-Muenzner-1}.

\begin{remark}[Consequences of Theorem \ref{thm:1-Muenzner-1}]
\label{rem:range-V}

\noindent
{\rm From equation (\ref{eq:1-Muenzner-diff-eq-S-1}), we see that the range of the restriction $V$ of $F$ to $S^{n+1}$ 
is contained in the closed interval $[-1, 1]$, since the left side of the equation is nonnegative.
We can see that the range of $V$ is all of the interval $[-1, 1]$ as follows.  Since $V$ is not constant 
on $S^{n+1}$, it has distinct maximum and minimum values on $S^{n+1}$.  By equation (\ref{eq:1-Muenzner-diff-eq-S-1})
these maximum and minimum values are 1 and $-1$, respectively, since ${\rm grad}^S V$ is nonzero at any point
where $V$ is not equal to $\pm 1$.  For any $s$ in the open interval $(-1, 1)$, the level set
$V^{-1} (s)$ is a compact hypersurface, since ${\rm grad}^S V$  is never zero on $V^{-1} (s)$.  M\"{u}nzner also
proves that each level set of $V$ is connected, and therefore, the original 
connected isoparametric hypersurface $M$ is contained in a unique compact, connected isoparametric hypersurface.

For $s= \pm 1$, ${\rm grad}^S V$ is identically equal to zero on $V^{-1} (s)$, and the sets 
$M_+ = M_{1} = V^{-1} (1)$ and
$M_- = M_{-1} = V^{-1} (-1)$ are submanifolds of codimension greater than one in $S^{n+1}$.  
M\"{u}nzner showed that $M_+$ and $M_-$ are connected (see Theorem \ref{thm:1-connectedness}), 
and that they are the focal submanifolds of any
isoparametric hypersurface $V^{-1} (s)$, $-1 < s < 1$, in the family of isoparametric hypersurfaces.  Thus,
there are only two focal submanifolds 
regardless of the number $g$ of distinct principal curvatures.  
By Theorem \ref{thm:1-prin-curv-isop-hyp}, there are $2g$ focal points evenly distributed along each normal geodesic
to the family $\{V^{-1} (s)\}$ of isoparametric hypersurfaces.  M\"{u}nzner proves in part (a) of
Theorem \ref{thm:1-ball-bundles} below
that these focal points lie alternately
on the two focal submanifolds $M_+$ and $M_-$.

Suppose we consider the isoparametric hypersurface $V^{-1} (0)$.  From equation (\ref{eq:1-V}), we see that the function
$\tau$ equals $\pi/2g$ on $V^{-1} (0)$.  The function $\tau$ is the distance from a point $x$ in $V^{-1} (0)$ to the
first focal point along the normal geodesic through $x$.  By Theorem \ref{thm:1-prin-curv-isop-hyp}, this
means that the largest principal curvature of $V^{-1} (0)$ is $\cot (\pi/2g)$, and the principal curvatures
of $V^{-1} (0)$ are given by $\cot \theta_i$, where
\begin{equation}
\label{eq:1-prin-curv-M_0}
\theta_i = \frac{\pi}{2g} + \frac{(i-1)}{g} \pi, \quad 1 \leq i \leq g,
\end{equation}
with multiplicities satisfying $m_{i+2} = m_i$ (subscripts mod $g$).
This simple form
for $\theta_i$ is helpful in the calculations involved in the proof of Theorem \ref{thm:1-Muenzner-1}.}
\end{remark}

\subsection{Global Structure Theorems}
\label{sec:1.6}
In this section, we discuss several results concerning the global structure of an isoparametric family of
hypersurfaces in $S^{n+1}$ that stem from M\"{u}nzner's construction of the Cartan-M\"{u}nzner polynomials,
as discussed in the previous subsection (see \cite[pp. 67--69]{Mu}).

These structure theorems ultimately lead to
M\"{u}nzner's main structure theorem which
states that a compact, connected isoparametric hypersurface $M$ divides the sphere $S^{n+1}$ into two
ball bundles over the two focal submanifolds $M_+$ and $M_-$, which lie on different sides of $M$ in $S^{n+1}$.

From this theorem, M\"{u}nzner ultimately uses methods from algebraic topology to
derive his primary result that the number $g$ of distinct principal curvatures of an isoparametric
hypersurface in $S^{n+1}$ must satisfy  $g =1,2,3,4$ or 6.

Let $F:{\bf R}^{n+2} \rightarrow {\bf R}$ be the Cartan-M\"{u}nzner polynomial of degree $g$ constructed
from a connected isoparametric hypersurface in $S^{n+1} \subset {\bf R}^{n+2}$ with $g$ distinct
principal curvatures as in Theorem \ref{thm:1-Muenzner-1}, and let $V$ denote the restriction of $F$
to $S^{n+1}$.  As noted in Remark \ref{rem:range-V},
the range of the function $V$ is the closed interval
$[-1,1]$, and each level set $M_t = V^{-1} (t),\  -1 < t < 1$, is an isoparametric hypersurface in $S^{n+1}$.

For the sake of definiteness, we let $M = M_0 = V^{-1} (0)$ be the isoparametric hypersurface discussed in 
Remark \ref{rem:range-V}. We denote the two (possibly equal) multiplicities of $M$ by $m_+ = m_1$
and $m_- = m_{-1}$.  

The first goal is to show that $M$ and all of the parallel hypersurfaces $M_t,\  -1 < t < 1$, as well as the two focal submanifolds, are connected.  We omit the proof here and refer the reader to M\"{u}nzner's paper \cite{Mu} or to
\cite[pp. 130--137]{CR8}.

\begin{theorem} [Connectedness of the level sets of $F$]
\label{thm:1-connectedness}
Let $F:{\bf R}^{n+2} \rightarrow {\bf R}$ be a Cartan-M\"{u}nzner polynomial of degree $g$ and $V$ its restriction
to $S^{n+1}$.  Then each isoparametric hypersurface
\begin{displaymath}
M_t = V^{-1} (t),\  -1 < t < 1,
\end{displaymath}
is connected.  Moreover,
$M_+ = V^{-1}(1)$ and $M_- = V^{-1}(-1)$ are the focal submanifolds,
and they are also connected.  
\end{theorem}

\begin{remark}
\label{rem:compact-connected}
{\rm A consequence of Theorem \ref{thm:1-connectedness} is that any connected isoparametric
hypersurface $M$ lies in a unique compact, connected isoparametric hypersurface of the form
$V^{-1}(t), -1 < t < 1$, where $V$ is the restriction to $S^{n+1}$ of the Cartan-M\"{u}nzner polynomial of $M$.
For the rest
of Section \ref{munzner-work}, we will assume that each isoparametric hypersurface and each focal submanifold is
compact and connected.}
\end{remark}

The next step is M\"{u}nzner's important structure theorem which
states that a compact, connected isoparametric hypersurface $M$ divides the sphere $S^{n+1}$ into two
ball bundles over the two focal submanifolds $M_+$ and $M_-$, which lie on different sides of $M$ in $S^{n+1}$.
The precise wording of the theorem is as follows. (See M\"{u}nzner \cite[pp. 67--69]{Mu} or \cite[pp. 132--133]{CR8} for a proof.)

\begin{theorem}
\label{thm:1-ball-bundles} 
Let $k= \pm1$, and let $Z$ be a normal vector to the focal submanifold $M_k$ in $S^{n+1}$. Let 
$\exp:NM_k \rightarrow S^{n+1}$ denote the normal exponential map for $M_k$.  Then
\begin{enumerate}
\item[${\rm(a)}$] $V(\exp Z) = k \cos(g |Z|)$.
\item[${\rm(b)}$] Let $B_k = \{q \in S^{n+1} \mid kV(q) \geq 0 \}$, and let $(B^\perp M_k,S^\perp M_k)$ be the bounded
unit ball bundle in $NM_k$.  Then
\begin{displaymath}
\psi_k : (B^\perp M_k,S^\perp M_k) \rightarrow (B_k,M),
\end{displaymath}
where $M=V^{-1}(0)$ and $\psi_k (Z) = \exp (\frac{\pi}{2g}Z)$, is a diffeomorphism of manifolds with boundary.
\end{enumerate}
\end{theorem}

As a consequence of the fact that the set of normal geodesics to each focal submanifold $M_k$, $k = \pm 1$, is the same 
as the set of normal geodesics to each of the parallel isoparametric hypersurfaces, we immediately obtain the following corollary.

\begin{corollary}
\label{cor:1-same-focal set} 
Let $M_k$, $k = \pm 1$, be a focal submanifold of an isoparametric hypersurface $M$.  Then the focal set of $M_k$ is 
the same as the focal set of $M$, i.e., it is $M_k \cup M_{-k}$.
\end{corollary}

\subsubsection{M\"{u}nzner's restriction on the number of principal curvatures}
\label{restriction-on-g}

M\"{u}nzner's major result
is that the number $g$ of distinct principal curvatures of an isoparametric hypersurface
$M$ in $S^{n+1}$ is $1,2,3,4$ or 6.  This is a lengthy and delicate computation involving the cohomology rings of the hypersurface $M$ and its two focal submanifolds $M_+$ and $M_-$.
The structure of these rings is determined by the basic topological fact that a compact,
connected isoparametric hypersurface $M \subset S^{n+1}$ divides $S^{n+1}$ into two ball bundles over the
two focal submanifolds, as in Theorem \ref{thm:1-ball-bundles} (b).    

Theorem \ref{thm:1-dim-cohomology} below does not assume that $M$ is isoparametric,
but only that it divides the sphere into two ball bundles.  This is important, since the theorem can be applied to
more general settings, in particular, to the case of a compact, connected proper Dupin hypersurface
embedded in $S^{n+1}$, as was shown by Thorbergsson \cite{Th1} (see also \cite[pp. 140--143]{CR8}).

\begin{remark}[Proper Dupin hypersurfaces]
\label{rem:def-dupin}
{\rm An oriented hypersurface $M$ in a real space form $\widetilde{M}^{n+1}$ is called a 
{\em proper Dupin hypersurface} if the number $g$ of distinct principal curvatures is constant on $M$, and if each continuous principal curvature function on $M$ is constant along each leaf of its corresponding principal foliation.}
\end{remark}

Using methods of algebraic topology,
M\"{u}nzner \cite{Mu2} proved the theorem below, and we refer the reader to M\"{u}nzner's paper for the proof.

\begin{theorem}
\label{thm:1-dim-cohomology} 
Let $M$ be a compact, connected hypersurface in $S^{n+1}$ which divides $S^{n+1}$ into two ball bundles
over submanifolds $M_+$ and $M_-$.  Then the number $\alpha = (1/2)\ {\rm dim}\ H^* (M,R)$ can only assume the
values $1,2,3,4$ and 6. (The ring $R$ of coefficients is ${\bf Z}$ if both $M_+$ and $M_-$ are orientable, 
and ${\bf Z}_2$ otherwise.)
\end{theorem}

M\"{u}nzner then proved Theorem \ref{thm:1-cohomology} below regarding the cohomology of an isoparametric hypersurface and its
focal submanifolds.  Since all of the parallel hypersurfaces $M_t = V^{-1}(t)$ are diffeomorphic, it is sufficient
to consider the case $M = V^{-1}(0)$.  In that case, $M$ has two focal submanifolds $M_1 = V^{-1}(1)$
of dimension $n - m_1$ and $M_{-1} = V^{-1}(-1)$ of dimension $n - m_{-1}$, where $m_1$ and $m_{-1}$ are the
two (possibly equal) multiplicities of the principal curvatures of $M$.  Then by Theorem \ref{thm:1-ball-bundles},
the sets,
\begin{displaymath}
B_1 = \{q \in S^{n+1} \mid V(q) \geq 0 \}, \quad B_{-1} = \{q \in S^{n+1} \mid V(q) \leq 0 \},
\end{displaymath}
are $(m_k + 1)$-ball bundles over the focal submanifolds $M_k$, for $k = 1, -1$, respectively.

The dimension $n$ of $M$ is equal to $g(m_1 + m_{-1})/2$, and so $g = 2n/\mu$, where $\mu = m_1 + m_{-1}$.
Thus, an isoparametric hypersurface $M$ satisfies the hypothesis of the following theorem of M\"{u}nzner \cite{Mu2}.
The proof is done using algebraic topology, and we will omit it here.  We refer the reader to
M\"{u}nzner's original proof \cite{Mu2} (see also \cite[pp. 134--136]{CR8}).

\begin{theorem}
\label{thm:1-cohomology} 
Let $M$ be a compact, connected hypersurface in $S^{n+1}$ such that:
\begin{enumerate}
\item[${\rm(a)}$] $S^{n+1}$ is divided into two manifolds $(B_1,M)$ and $(B_{-1},M)$ with boundary along $M$.
\item[${\rm(b)}$]  For $k = \pm 1$, the manifold $B_k$ has the structure of a differentiable ball bundle over a
compact, connected manifold $M_k$ of dimension $n-m_k$.
\end{enumerate}
Let the ring $R$ of coefficients be ${\bf Z}$ if both $M_1$ and $M_{-1}$ are orientable, and ${\bf Z}_2$ otherwise.
Let $\mu = m_1 + m_{-1}$.  Then $\alpha = 2n/\mu$ is an integer, and for $k = \pm 1$,
\begin{displaymath}
H^q(M_k) = \left\{ \begin{array}{ll}
R & \mbox{for }q \equiv 0 \ (\bmod \  \mu),\ 0 \leq q <n,\\
R & \mbox{for } q \equiv m_{-k} \ (\bmod \ \mu),\ 0 \leq q <n,\\
0 & \mbox{otherwise}.
\end{array}
\right. 
\end{displaymath}
Further,
\begin{displaymath}
H^q(M) = \left\{ \begin{array}{l}
R \ \mbox{for }q = 0, n, \\
H^q(M_1) \oplus H^q (M_{-1}), \ \mbox{for }1 \leq q \leq n-1.
\end{array}
\right. 
\end{displaymath}
\end{theorem}

For a compact, connected isoparametric hypersurface $M \subset S^{n+1}$ with $g$ distinct principal curvatures, we have
\begin{displaymath}
\dim M = n = g(m_1 + m_{-1})/2 = g\mu/2.
\end{displaymath} 
Thus, $\alpha = 2n/\mu = g$.  By Theorem \ref{thm:1-cohomology}, we see that 
$\alpha$ is also equal to $\dim_R H^*(M,R)/2$.  Hence by M\"{u}nzner's Theorem \ref{thm:1-dim-cohomology}, 
the number $g = \alpha$
can only assume the values $1,2,3,4$ or 6, and we have M\"{u}nzner's \cite{Mu2} major theorem.
\begin{theorem}
\label{thm:1-number-prin-curv} 
Let $M \subset S^{n+1}$ be a connected isoparametric hypersurface with $g$ distinct principal curvatures.
Then $g$ is $1,2,3,4$ or 6. 
\end{theorem}

Note that we do not have to assume that $M$ is compact in the theorem, because any connected isoparametric 
hypersurface is contained in a unique compact, connected isoparametric hypersurface to which the arguments above can
be applied (see Remark \ref{rem:compact-connected}).  We also note that
there exist isoparametric hypersurfaces for each of the values of $g$ in the theorem, as mentioned in Subsection 
\ref{cartan-questions}.

\begin{remark}[Crystallographic groups]
\label{rem:crystall}

\noindent
{\rm A consequence of M\"{u}nzner's Theorem \ref{thm:1-number-prin-curv} is that the dihedral group
$D_g$ associated to $M$ (see Remark \ref{rem:Coxeter-group}) is crystallographic
(see L.C. Grove and C.T. Benson \cite[pp. 21--22]{Grove-Benson}). 
A direct proof that $D_g$ must be crystallographic could possibly give a simpler proof of Theorem \ref{thm:1-number-prin-curv} (see also K. Grove and S. Halperin \cite[pp. 437--438]{GH}).}
\end{remark}

\subsubsection{Multiplicities of the principal curvatures}
\label{multiplicities-prin-curv}

As discussed in Subsection \ref{cartan-case g=3},
Cartan \cite{Car3} classified isoparametric hypersurfaces with $g \leq 3$ principal curvatures.
 In the cases $g=4$ and $g=6$, many results concerning the possible
multiplicities of the principal curvatures
have been obtained from the topological situation described in Theorem \ref{thm:1-cohomology}, i.e.,
that a compact, connected isoparametric hypersurface $M$ in $S^{n+1}$ divides $S^{n+1}$ into two ball bundles
over its two focal submanifolds.  

In the case of $g=4$ principal curvatures, 
several mathematicians, including M\"{u}nzner \cite{Mu}--\cite{Mu2},
Abresch \cite{Ab}, Grove and Halperin \cite{GH}, Tang \cite{Tang}--\cite{Tang2} 
and Fang \cite{Fang3}--\cite{Fang1}, 
found restrictions on the multiplicities $(m_1,m_2)$. 
This series of results culminated with the paper of
Stolz \cite{Stolz}, who used techniques from homotopy theory to prove Theorem \ref{Stolz-isop} below.  
This theorem of Stolz was ultimately an important part in the proof of the classification of isoparametric hypersurfaces with $g=4$ principal curvatures, which will be discussed in Section \ref{chap:2}.

\begin{theorem} [Stolz]
\label{Stolz-isop}
The multiplicities $(m_1,m_2)$ of the principal curvatures of an isoparametric hypersurface  $M \subset S^{n+1}$
with $g=4$ principal curvatures
are the same as those in the examples due to Ferus, Karcher and M\"{u}nzner \cite{FKM}
or the two homogeneous examples with $(m_1,m_2)= (2,2)$ or $(m_1,m_2)=(4,5)$ that are not of FKM-type.
\end{theorem}

\begin{remark}[Application to proper Dupin hypersurfaces]
\label{rem:Stolz-Dupin}
{\rm Stolz proved Theorem \ref{Stolz-isop} under more general assumption that $M$ is
a compact, connected proper Dupin
hypersurface with four principal curvatures embedded in $S^{n+1}$.  Such a result is possible because
Thorbergsson \cite{Th1} had shown earlier that
a compact, connected proper Dupin hypersurface $M \subset S^{n+1}$ separates 
$S^{n+1}$ into two ball bundles over the first focal submanifolds on either side of $M$,
as in the case of a compact, connected isoparametric hypersurface (see also
\cite [p. 143]{CR8}). }
\end{remark}

In the case of $g=6$ principal curvatures, M\"{u}nzner \cite{Mu2} showed that all of the principal curvatures
have the same multiplicity $m$, and Abresch \cite{Ab} showed that $m$ equals 1 or 2.
Thus we have:

\begin{theorem}
\label{Abresch}
For an isoparametric hypersurface with $g=6$ principal curvatures, all the principal curvatures have the same multiplicity $m$, and $m$ equals 1 or 2.
\end{theorem}

 Later Grove and Halperin  \cite{GH} showed that this restriction on the multiplicities also holds for
compact, connected proper Dupin
hypersurfaces with $g=6$ principal curvatures embedded in $S^{n+1}$.

\section{Classification Results}
\label{chap:2}

Due to the work of many mathematicians, isoparametric hypersurfaces in $S^{n+1}$ have been completely classified, and we will briefly state the classification results now.  As noted above, M\"{u}nzner
proved that the number $g$ of distinct principal curvatures must be $1,2,3, 4$ or 6.
We now summarize the classification results for each value of $g$. Many of these have been discussed already.

If $g=1$,
then the isoparametric hypersurface $M$ is totally umbilic, and it must be an open subset of a great or small
sphere.  If $g=2$, Cartan \cite{Car2} showed that $M$ must be an open subset of a standard product of two spheres,
\begin{displaymath}
S^k(r) \times S^{n-k-1}(s) \subset S^n, \quad r^2+s^2=1.
\end{displaymath}

In the case $g=3$, Cartan \cite{Car3}
showed that all the principal curvatures must have the same multiplicity
$m=1,2,4$ or 8, and the isoparametric hypersurface must be an open subset of a tube of
constant radius 
over a standard embedding of a projective
plane ${\bf FP}^2$ into $S^{3m+1}$,
where ${\bf F}$ is the division algebra
${\bf R}$, ${\bf C}$, ${\bf H}$ (quaternions),
${\bf O}$ (Cayley numbers), for $m=1,2,4,8,$ respectively.  Thus, up to
congruence, there is only one such family for each value of $m$. (See Subsection \ref{cartan-case g=3}.)

In the case of an 
isoparametric hypersurface with $g=4$ four principal curvatures, 
M\"{u}nzner (Theorem \ref{thm:1-prin-curv-isop-hyp}) proved that
the principal curvatures can have at most two
distinct multiplicities $m_1,m_2$. 
Next Ferus, Karcher and M\"{u}nzner
\cite{FKM} 
used representations
of Clifford algebras to construct
for any positive integer
$m_1$ an infinite series of families of isoparametric hypersurfaces
with four principal curvatures having respective multiplicities
$(m_1,m_2)$, where $m_2$ is nondecreasing and
unbounded in each series.  These examples are now known as isoparametric hypersurfaces of {\it FKM-type}, and
they are also described in
 \cite[pp. 162--180]{CR8}.  
This construction of Ferus, Karcher and M\"{u}nzner was a generalization of an earlier construction due to Ozeki and Takeuchi \cite{OT}.

Stolz \cite{Stolz} (see Theorem \ref{Stolz-isop}) next proved that the multiplicities $(m_1,m_2)$
of the principal curvatures of an isoparametric
hypersurface with $g=4$ principal curvatures
must be the same as those of the hypersurfaces of
FKM-type, or else $(2,2)$ or $(4,5)$, which are the multiplicities for certain homogeneous examples
that are not of FKM-type  (see the classification 
of homogeneous isoparametric hypersurfaces of Takagi and Takahashi \cite {TT}).

Cecil, Chi and Jensen \cite{CCJ1} then
showed that if the multiplicities of an isoparametric hypersurface with four principal curvatures satisfy the condition
$m_2 \geq 2 m_1 - 1$, then
the hypersurface is of FKM-type.  (A different proof of this result, using isoparametric 
triple systems, was given later by Immervoll \cite{Im}.)

Taken together with known results of Takagi \cite{Takagi} for $m_1 = 1$,
and Ozeki and Takeuchi \cite{OT} for $m_1 = 2$, this result of Cecil, Chi and Jensen handled all
possible pairs of multiplicities except for four cases, the FKM pairs
$(3,4), (6,9)$ and $(7,8)$,and  the homogeneous pair $(4,5)$.

In a series of recent papers, Chi \cite{Chi}--\cite{Chi4} completed the classification of isoparametric hypersurfaces with four principal curvatures.
Specifically, Chi showed that in the cases of multiplicities
$(3,4)$, $(6,9)$ and $(7,8)$, the isoparametric hypersurface must be of FKM-type, and in the case $(4,5)$, it must be homogeneous. 

The final conclusion is that an isoparametric hypersurface with $g=4$ principal curvatures must either be of
FKM-type, or else a homogeneous isoparametric hypersurface with multiplicities $(2,2)$ or $(4,5)$.

In the case of an isoparametric hypersurface with $g=6$ principal curvatures, M\"{u}nzner \cite{Mu}--\cite{Mu2} showed
that all of the principal curvatures must have the same multiplicity $m$, and
Abresch \cite{Ab} showed that $m$ must equal 1 or 2.
By the classification of homogeneous isoparametric hypersurfaces due to Takagi and Takahashi \cite{TT},
there is up to congruence only one homogeneous family in each case, $m=1$ or $m=2$.  

These 
homogeneous examples have been shown to be
the only isoparametric hypersurfaces in the case $g=6$ by Dorfmeister and Neher \cite{DN5}
in the case of multiplicity $m=1$, and by Miyaoka \cite{Mi11}--\cite{Mi12} in the case $m=2$
(see also the papers of Siffert \cite{Siffert1}--\cite{Siffert2}).

For general surveys on isoparametric hypersurfaces in spheres,
see the papers of Thorbergsson \cite{Th6}, Cecil \cite{Cec9}, and Chi \cite{Chi-survey}.

\begin{remark}
\label{rem-isop-complex-space-forms-2}
{\rm Although, as noted in Remark \ref{rem-isop-complex-space-forms}, isoparametric hypersurfaces in the complex space forms are related to those in the real case (sphere and anti-de Sitter space), classification questions present further difficulties.  First, the number $g$ of distinct principal curvatures need not be constant.  Second, all hypersurfaces have a distinguished unit vector field $W = -J \xi$, where $\xi$ is the unit normal.  $W$ is called the structure vector field (or sometimes, the Hopf vector field).  If $W$ is a principal vector then $M$ is a Hopf hypersurface.  Typically, however, $W$ may have nonzero components in more than one eigenspace of the shape operator $A$.  The number $h$ of such components also figures into the analysis of isoparametric hypersurfaces and, like $g$, need not be constant.

Although a description of the classification results is beyond the scope of this paper, we can say the following:

\begin{theorem}
For a hypersurface $M$ in ${\bf CP}^n$, the following are equivalent
\begin{enumerate}
\item
$M$ is isoparametric and Hopf;
\item
$M$ is Hopf and has constant principal curvatures;
\item
$M$ is isoparametric and has constant principal curvatures;
\item
$M$ is an open subset of a homogeneous hypersurface (see Takagi's list \cite[p. 350]{CR8}).
\end{enumerate} 
\end{theorem}

\begin{theorem}
Let $M$ be an isoparametric hypersurface in ${\bf CP}^n$.  For each point $p \in M$, we have $h(p) \in \{1,3,5\}$ and $g(p) \in \{2,3,4,5,7\}$.
\end{theorem}

\begin{theorem}
Let $M$ be an isoparametric hypersurface in ${\bf CH}^n$.  For each point $p \in M$, we have $h(p) \in \{1,2,3\}$ and $g(p) \in \{2,3,4,5\}$.
\end{theorem} }
\end{remark}

\noindent Thomas E. Cecil

\noindent Department of Mathematics and Computer Science

\noindent College of the Holy Cross

\noindent Worcester, Massachusetts 01610, U.S.A.

\noindent email: tcecil@holycross.edu\\

\noindent Patrick J. Ryan

\noindent Department of Mathematics and Statistics

\noindent McMaster University

\noindent Hamilton, Ontario, Canada L8S4K1

\noindent email: ryanpj@mcmaster.ca

\end{document}